\documentclass[11pt]{article}
\usepackage{amsfonts,amssymb,amsmath,amsthm}

\usepackage{bm}
\usepackage{palatino}
\usepackage{mathpazo}
\usepackage{inconsolata}
\usepackage[hidelinks]{hyperref}

\usepackage{xcolor}
\usepackage{varwidth}
\usepackage{multicol}

\usepackage{tilings}
\usetikzlibrary{patterns}
\usetikzlibrary{calc}
\tikzset{invisible/.style={minimum width=0mm,inner sep=0mm,outer sep=0mm}}
\tikzset{tiling/.style={invisible, anchor=north west}}
\tikzset{ptnode/.style={rounded corners=15, black!40, thick, dashed}}

\begingroup
    \makeatletter
    \@for\theoremstyle:=definition,plain\do{%
        \expandafter\g@addto@macro\csname th@\theoremstyle\endcsname{%
            \addtolength\thm@preskip\parskip}%
        }
\endgroup

\usepackage[margin=1in]{geometry}

\theoremstyle{definition}
\newtheorem{theorem}{Theorem}
\newtheorem{proposition}[theorem]{Proposition}
\newtheorem{lemma}[theorem]{Lemma}

\newtheorem{remark}[theorem]{Remark}

\numberwithin{theorem}{section}
\numberwithin{question}{section}
\numberwithin{definition}{section}
\numberwithin{example}{section}

\newcommand*\patchAmsMathEnvironmentForLineno[1]{%
  \expandafter\let\csname old#1\expandafter\endcsname\csname #1\endcsname
  \expandafter\let\csname oldend#1\expandafter\endcsname\csname end#1\endcsname
  \renewenvironment{#1}%
     {\linenomath\csname old#1\endcsname}%
     {\csname oldend#1\endcsname\endlinenomath}}%
\newcommand*\patchBothAmsMathEnvironmentsForLineno[1]{%
  \patchAmsMathEnvironmentForLineno{#1}%
  \patchAmsMathEnvironmentForLineno{#1*}}%
\AtBeginDocument{%
\patchBothAmsMathEnvironmentsForLineno{equation}%
\patchBothAmsMathEnvironmentsForLineno{align}%
\patchBothAmsMathEnvironmentsForLineno{flalign}%
\patchBothAmsMathEnvironmentsForLineno{alignat}%
\patchBothAmsMathEnvironmentsForLineno{gather}%
\patchBothAmsMathEnvironmentsForLineno{multline}%
}

\usepackage{parskip}
\usepackage{lineno}
\usepackage[small]{titlesec}

\usepackage[square,comma,numbers,sort&compress]{natbib}

\usepackage{color}
\definecolor{todocolor}{RGB}{205,235,139}
\definecolor{todo-idea}{RGB}{120,180,255}
\definecolor{todo-error}{RGB}{208,31,60}
\definecolor{todo-question}{RGB}{255,255,136}

\usepackage[colorinlistoftodos, color=todocolor, textsize=small]{todonotes}
\usepackage{colortbl}
\usepackage{tabularray}

\usepackage{ stmaryrd }
\usepackage{tikz}
\usepackage{pgfplots}
\pgfplotsset{width=10cm,compat=1.9}

\newcommand{\Av}{\operatorname{Av}}

\newcommand{\CC}{\mathcal{C}}

\newcommand{\Grid}{\operatorname{Grid}}

\newcommand{\SB}{\mathcal{SB}}

\newcommand{\SBH}{\mathcal{HSB}}

\renewenvironment{abstract}{
	\begin{list}{}%
	{\setlength{\rightmargin}{1in}%
	\setlength{\leftmargin}{1in}}%
	\item[]\ignorespaces\begin{small}}%
	{\end{small}\unskip\end{list}%
}

\newpagestyle{main}[\small]{
	\headrule{}
	\sethead[\usepage][][]
	{\sc }{}{\usepage}
}

\title{The Insertion Encoding of Cayley Permutations}

\usepackage{makecell}
\author{
	\begin{tabular}{c}
		\makecell{
			Christian Bean\\
			\small School of Computer Science\\
			\small and Mathematics\\
			\small Keele University\\
			\small Keele, United Kingdom\\
			\small \texttt{c.n.bean@keele.ac.uk}
		}
		\quad\quad
		\makecell{
			Paul C. Bell\\
			\small School of Computer Science\\
			\small and Mathematics\\
			\small Keele University\\
			\small Keele, United Kingdom\\
			\small \texttt{p.c.bell@keele.ac.uk}
		}\\ \\
		\makecell{
			Abigail Ollson\\
			\small School of Computer Science\\
			\small and Mathematics\\
			\small Keele University\\
			\small Keele, United Kingdom\\
			\small \texttt{a.n.ollson@keele.ac.uk}
		}
	\end{tabular}
}

\date{}

\usepackage{arydshln}
\usepackage{makecell}
\usepackage{mathtools} 
\DeclarePairedDelimiter{\ceil}{\lceil}{\rceil} 
\usepackage{graphicx} 
\usepackage{subcaption} 

\begin{document}
\maketitle

\begin{abstract}
	We introduce the vertical and horizontal insertion encodings for Cayley permutations which naturally generalise the insertion encoding for permutations. In both cases, we fully classify the Cayley permutation classes for which these languages are regular, and provide an algorithm for computing the rational generating functions. We use our algorithm to solve an open problem of Cerbai by enumerating the hare pop-stack sortable Cayley permutations.
\end{abstract}


\section{Introduction}
\label{sec:intro}

A \emph{Cayley permutation} is a word $\pi \in \mathbb{N}_1^*$ such that every number between $1$ and the maximum value of $\pi$ appears at least once. 
The plot of a Cayley permutation $\pi_1\pi_2 \cdots \pi_n$ is the set of points $(i, \pi_i)$. For example, the plot of the Cayley permutation $21321$ is shown in Figure~\ref{fig:21321}. 

The set of Cayley permutations is in bijection with the set of ordered set partitions; the value $i$ in the $j^{th}$ block of the ordered set partition implies the $i^{th}$ index of the Cayley permutation has value $j$. For example, the Cayley permutation $21321$ corresponds to the ordered set partition $\{2, 5\}, \{1, 4\}, \{3\}$. Therefore, the number of Cayley permutations of size $n$ is the $n^{th}$ Fubini number, A000670 in the Online Encyclopedia of Integer Sequences (OEIS)~\cite{oeis}.

\begin{figure}[h]
    \centering
    \resizebox{4cm}{!}{
        \begin{tikzpicture}
            \draw[step=1cm,gray,thin] (0.5,0.5) grid (5.5,3.5);
            \foreach \x/\y in {1/2, 2/1, 3/3, 4/2, 5/1}
            \fill[black] (\x,\y) circle (0.13cm);
        \end{tikzpicture}
    }
    \caption{The plot of the Cayley permutation $21321$.}
    \label{fig:21321}
\end{figure}

There are many objects which are equivalent to Cayley permutations, for example surjective words~\cite{Hazewinkel2004}, packed words~\cite{Foissy2010,Kroes2022}, initial words~\cite{Patrias2016}, and preferential arrangements~\cite{Ahlbach2012,BiersAriel2016, Spippenger2010,Nkonkobe2020}. Cayley permutations were first defined in 1983 by Mor and Fraenkel~\cite{Mor1984}. The name was inspired by a 1857 paper of Cayley~\cite{Cayley1857} where he counted a sub-class of trees which are enumerated by the ordered Bell numbers, also known as the Fubini numbers. Mor and Fraenkel called these trees \emph{Cayley trees} and found that they are in bijection with Cayley permutations. They also gave a new formula for enumerating Cayley permutations by counting the number of nondecreasing Cayley permutations and then counting the number of ways in which each of these can be permuted. 

The \emph{standardisation} of a word in $\mathbb{N}_1^*$ is the Cayley permutation obtained by replacing every occurrence of the smallest integer by $1$, the next smallest by $2$, and so on. For example, the standardisation of $4674$ is $1231$. We say that the word $w = w_1w_2\cdots w_n$ \emph{contains} the Cayley permutation $p = p_1p_2\cdots p_k$ if there exist indices $i_1 < i_2 < \cdots < i_k$ such that the standardisation of $w_{i_1}w_{i_2}\cdots w_{i_k}$ is $p$. We say $w$ \emph{avoids} $p$ if $w$ does not contain $p$. In this context, we call $p$ a pattern. For example, the Cayley permutation $21321$ contains multiple occurrences of the pattern $121$, such as the standardisation of $232$ or an occurrence of 121 directly, but $21321$ avoids the pattern $123$ as there is no strictly increasing subsequence of size 3 in $21321$.

This containment order forms a partial order on the set of Cayley permutations. A set of Cayley permutations that is closed downwards is called a \emph{Cayley permutation class}, i.e., if $\pi \in \CC$ and $\pi$ contains $\sigma$, then $\sigma \in \CC$. A Cayley permutation class can be uniquely defined by the set of minimal Cayley permutations not in $\CC$, called the \emph{basis}. For a set of Cayley permutations $B$ we say that a Cayley permutation $\pi$ \emph{contains} $B$ if $\pi$ contains any element of $B$, and \emph{avoids} $B$ if $\pi$ avoids every element of $B$. The set of Cayley permutations that avoid $B$ is denoted $\Av(B)$.

Cayley permutations can be seen as a generalisation of permutations where repeated values are allowed, hence the Cayley permutation class $\Av(11)$ describes all permutations. The theory of permutation classes has been very well developed. One of the first noteable results was by MacMahon in 1915~\cite{MacMahon1915} where he found that the Catalan numbers enumerate the class of permutations avoiding 123. In 1968, Knuth~\cite{Knuth68} asked as an exercise to prove that a permutation is stack sortable if and only if it avoids $231$. The permutation class $\Av(231)$ is also enumerated by the Catalan numbers. These classes can be thought of as the Cayley permutation classes $\Av(11,123)$ and $\Av(11,231)$ respectively. 
Knuth's work brought much more focused attention to the area of pattern-avoiding permutations eventually leading to the first systematic study of pattern-avoiding permutations in 1985 by Simion and Schmidt~\cite{Simion1985}. For a more detailed overview see for example the survey by Vatter~\cite{vatter2015permutation}.

The goal of this work is to extend some of the methods that have been developed for permutation classes to study Cayley permutation classes. In particular, we extend the insertion encoding of permutations introduced by Albert, Linton and Ru\v{s}kuc~\cite{Albert2005} to Cayley permutations in two different ways; first, by inserting new maxima, which we will call the \emph{vertical insertion encoding} and then by inserting new rightmost values, which we will call the \emph{horizontal insertion encoding}. As permutations are closed with respect to the inverse symmetry, inserting a new maximum is equivalent to inserting a new rightmost value so these two methods are equivalent. For Cayley permutations this is no longer the case. 
Section~\ref{sec:insertionencoding} and Section~\ref{sec:regular-theory} explain the method for the vertical insertion encoding and Section~\ref{sec:left_to_right} considers the horizontal insertion encoding. 

Our main results are in Section~\ref{sec:regular-theory} and Section~\ref{sec:left_to_right} which fully classify the Cayley permutation classes for which the vertical and horizontal insertion encodings form a regular language and follow a similar theory to the theory for permutations. We outline a method for computing these languages for the vertical insertion encoding in Section~\ref{sec:regular-algo}, which we also adapted to the horizontal case. Our implementation of these methods has been made available on GitHub~\cite{Bean_cperms_ins_enc_2025}.

Table~\ref{tab:results for size 3} shows the number of Cayley permutation classes defined by avoiding size 3 patterns which have either a set of vertical or horizontal insertion encodings which form a regular language. Every class defined by avoiding nine or more size three patterns has either a set of vertical or horizontal insertion encodings which is regular.
Of the 8191 bases containing Cayley permutations of size 3 only, 7498 of them have a set of vertical or horizontal insertion encodings which is regular.

\begin{table}[h]
    \centering
    \begin{tabular}{|c|c|c|c|c|}
        \hline
                      & Number     & Set of vertical     & Set of horizontal   & Either vertical         \\
        Size of basis & of classes & insertion encodings & insertion encodings & or horizontal insertion \\
                      &            & are regular         & are regular         & encodings are regular   \\
        \hline
        1             & 13         & 0                   & 0                   & 0                       \\
        2             & 78         & 0                   & 13                  & 13                      \\
        3             & 286        & 87                  & 111                 & 145                     \\
        4             & 715        & 435                 & 428                 & 528                     \\
        5             & 1287       & 1028                & 986                 & 1124                    \\
        6             & 1716       & 1550                & 1513                & 1625                    \\
        7             & 1716       & 1645                & 1631                & 1687                    \\
        8             & 1287       & 1269                & 1267                & 1283                    \\
        9             & 715        & 713                 & 713                 & 715                     \\        
        \hline
    \end{tabular}
    \caption{The number of Cayley permutation classes defined by avoiding a set of size 3 patterns. In the third and fourth columns, the number of classes with a set of vertical insertion encodings and a set of horizontal insertion encodings which are regular are shown. The last column shows the number of classes with either a set of vertical or horizontal insertion encodings which are regular.}
    \label{tab:results for size 3}
\end{table}

Recently, Claesson, Cerbai, Ernst and Golab~\cite{Claesson2024} enumerated $\Av(111)$, $\Av(112)$, $\Av(121)$, $\Av(123)$ and $\Av(132)$ using combinatorial species. Golab's thesis~\cite{Golab2024} went on to enumerate $\Av(112, 121)$ and $\Av(111, 112)$. None of these classes have either a set of vertical or horizontal insertion encodings which are regular.

Defant and Kravitz~\cite{Defant2018} have adapted the theory behind stack sortable permutations to Cayley permutations. 
They defined two different types of stack sorting, a hare-stack sort and a tortoise-stack sort. In a hare-stack sort, consecutive repeated values are allowed in the stack and in a tortoise-stack sort they are not.
They identified the classes of Cayley permutations that are hare-stack sortable and tortoise-stack sortable as $\Av(231)$ and $\Av(221,231)$ respectively, and found a recurrence relation for their counting sequences.
Cerbai~\cite{Cerbai2021} then found the bases for tortoise pop-stack sortable Cayley permutations to be $\Av(211, 221, 231, 312)$ and hare pop-stack sortable to be $\Av(231, 312, 2121)$, enumerating $\Av(211, 221, 231, 312)$. Using the insertion encoding, implemented using Tilings described in Section~\ref{sec:regular-algo}, we present the generating function for $\Av(231, 312, 2121)$ in Section~\ref{sec:algo-code}.

\section{The insertion encoding}
\label{sec:insertionencoding}

Albert, Linton and Ru\v{s}kuc~\cite{Albert2005} introduced the \emph{insertion encoding} for permutations and studied he languages which these form, which generalises the finitely labelled generating trees of  Chung, Graham, Hoggatt, and Kleiman~\cite{Chungetal}. The insertion encoding encodes how a permutation is built up by iteratively adding new maxima. The set of all insertion encodings of permutations in a class forms a language. Albert, Linton and Ru\v{s}kuc~\cite{Albert2005} studied these languages for permutation classes, in particular, classifying for which permutation classes the set of insertion encodings forms a regular language. They also gave a linear time algorithm in the size of the basis elements for checking whether a finitely based permutation class has a set of insertion encodings which form a regular language. Vatter~\cite{Vatter2009} gave an algorithm for computing a Deterministic Finite Automaton (DFA) for these regular languages which could then be used to compute the rational generating function of the permutation classes. An alternative algorithm using the tilescope algorithm~\cite{combexp} was later given, which experimentally appears to be significantly faster.

Unlike permutations, Cayley permutations can have repeated numbers, so in order to generalise this process to build Cayley permutations we will need to allow the insertion of a repeated maximum number. To ensure that it is unique we will build the Cayley permutations by inserting repeated values from left to right, which we call the \emph{vertical insertion encoding} in Section~\ref{sec:intro}. Choosing leftmost values was an arbitrary decision and would work just as well inserting repeated values from right to left.

Tracing how to build a Cayley permutation in this way starting from the empty Cayley permutation is called the \emph{evolution} of a Cayley permutation. In the evolution, we keep track of where future points will be inserted using a symbol `$~\diamond~$' called a \emph{slot}. For example, the evolution of the Cayley permutation $31232$ is shown in Figure~\ref{fig:evolution}.

\begin{figure}[h]
    \centering
    \begin{tabular}{cc}
         & $~\diamond~$                   \\
         & $~\diamond~ 1 ~\diamond~$      \\
         & $~\diamond~ 1~2~\diamond~$     \\
         & $~\diamond~ 1~2 ~\diamond~ 2~$ \\
         & $~3~1~2 ~\diamond~ 2~$         \\
         & $~3~1~2~3~2~$
    \end{tabular}
    \caption{The evolution of the Cayley permutation $31232$.}
    \label{fig:evolution}
\end{figure}

Each line of the evolution, such as the string $~\diamond~ 12~\diamond~ $, is called a \emph{configuration}.

For a configuration, we index the slots from left to right, starting with $1$. An insertion is determined by three things: the index of the slot inserted into, the way the slot is affected by the insertion, and whether or not the insertion is a repeated current maximum value.

There are four possibilities for how $n$ can be inserted into a slot, which we will represent with the letters $\ell, m, r, f$ denoting left, middle, right and fill respectively. Figure~\ref{tab:letters} shows the four ways to replace a slot with a combination of the number $n$ and some new slots.
\begin{figure}
    \centering
    \begin{tabular}{cccc}
        $\ell$: & $~\diamond~$ & $\rightarrow$ & $n ~\diamond~$            \\
        $m$:    & $~\diamond~$ & $\rightarrow$ & $~\diamond~ n ~\diamond~$ \\
        $r$:    & $~\diamond~$ & $\rightarrow$ & $~\diamond~ n$            \\
        $f$:    & $~\diamond~$ & $\rightarrow$ & $n$                       \\
    \end{tabular}
    \caption{The four ways to insert a number $n$ into a slot.}
    \label{tab:letters}
\end{figure}

To encode Cayley permutations, for each insertion we will use letters of the form $a_{i,j}$ with $a \in \{\ell,r,m,f\}$, $i \in \mathbb{N}_1$ is the index of the slot being inserted into, and $j \in \{0, 1\}$ tell us if we are inserting a repeated maximum value (0) or a new maximum value (1). Returning to our example $31232$, the vertical insertion encoding for this Cayley permutation will be the word $m_{1,1}\ell_{2, 1}r_{2, 0}f_{1,1}f_{1,0}$, shown in Figure~\ref{fig:example_encoding}.

\begin{figure}[h]
    \centering
    \begin{tabular}{ccc}
         & $~\diamond~$                   &               \\
         & $~\diamond~ 1 ~\diamond~$      & $m_{1, 1}$    \\
         & $~\diamond~ 1~2 ~\diamond~ $   & $\ell_{2, 1}$ \\
         & $~\diamond~ 1~2 ~\diamond~ 2~$ & $r_{2, 0}$    \\
         & $~3~1~2 ~\diamond~ 2~$         & $f_{1, 1}$    \\
         & $~3~1~2~3~2~$                  & $f_{1, 0}$
    \end{tabular}
    \caption{The evolution of a Cayley permutation 31232 with the encoding shown with each insertion.}
    \label{fig:example_encoding}
\end{figure}

For a given configuration, not every slot can be inserted into with a repeated value to ensure a unique evolution for every Cayley permutation. For example, for the configuration $~\diamond~ 1$ only new maxima can be inserted, so the only valid letters are $\ell_{1,1}, m_{1,1}, r_{1,1}$ and $f_{1,1}$ giving the configurations $2~\diamond~ 1 $, $~\diamond~ 2 ~\diamond~ 1$, $~\diamond~ 2 1$ and $21$ respectively.
But for the configuration $1 ~\diamond~$, either a new maximum or a repeated value can be inserted into the slot so any of the letters $\ell_{1,0}, m_{1,0}, r_{1,0}, f_{1,0}$, $\ell_{1,1}, m_{1,1}, r_{1,1}$ and $f_{1,1}$ can be applied, resulting in the configurations $1 1~\diamond~$, $1 ~\diamond~ 1 ~\diamond~$, $1 ~\diamond~ 1$, $11$, $1 2~\diamond~$, $1 ~\diamond~ 2 ~\diamond~$, $1 ~\diamond~ 2$ and $12$ respectively.

For a configuration $c$ we call the \emph{derivations} of $c$ the set of all Cayley permutations which have $c$ in its evolution.

\subsection{An example - the class $\Av(211, 312)$}
\label{sec:Av(312,212)}
In this subsection we will show how to enumerate a class of Cayley permutations using the vertical insertion encoding. 
In their paper on the insertion encoding of permutations, Albert, Linton and Ru\v{s}kuc~\cite{Albert2005} prove that the permutation class $\Av(11, 312)$ is counted by the Catalan numbers by finding a context-free grammar for it. They also prove that every insertion is in the first slot of a configuration in the evolution of the permutations in $\Av(11, 312)$.
We will prove in Proposition~\ref{prop:312_212_insertion_encoding} that this is also true of the Cayley permutation class $\Av(211, 312)$ for the vertical insertion encoding and show that is context-free by giving a non-ambiguous grammar for it.

\begin{proposition}
      \label{prop:312_212_insertion_encoding}
      A Cayley permutation avoids 211 and 312 if and only if its vertical insertion encoding only uses letters of the form $a_{1, j}$ for $a \in \{\ell, r, m, f\}$ and $j \in \{0, 1\}$.
\end{proposition}
\begin{proof}
      Only using letters of this form is equivalent to every insertion being into the first slot of the configuration.
      We will prove the statement by proving the contrapositive of each direction, first that if the evolution of a Cayley permutation inserts into a slot that is not the first slot then it contains 211 or 312. 
      Let $\pi$ be a Cayley permutation whose evolution contains an insertion into a slot that is not the first slot. Suppose this slot has index $i > 1$.
      As there exists at least one value between any two slots in a configuration, there is at least one value between the first slot and the $i^{th}$ slot. This value, together with the value inserted in the $i^{th}$ slot, will form an occurrence of either $12$ or $11$. 
      
      At some point in the evolution, the first slot will be inserted into. The value inserted into the first slot will be strictly larger than all values to the right. As there is an occurrence of either $12$ or $11$ to the right of the first slot, inserting this new value will create an occurrence of either $211$ or $312$. As $\pi$ contains either 211 or 312, $\pi \not\in \Av(211, 312)$.
      
      We now prove that if a Cayley permutation contains 211 or 312 then its evolution inserts into a slot that is not the first slot. Suppose $\pi$ is a Cayley permutation which contains 211 or 312.
      In both 211 and 312, the middle 1 is the leftmost smallest value so if $\pi$ contains either of these patterns then this middle 1 of any occurrence was inserted first in the evolution of $\pi$.
      
      Suppose in the evolution of $\pi$ some value $n$ is inserted to create a configuration $c$ and this occurrence of $n$ is at index $j$ in $c$.
      For this to later play the role of the middle 1 in either 211 or 312, there must be a slot at an index smaller than $j$ in $c$.
      By the rule of inserting smallest leftmost values first, the next value to be inserted for either 211 or 312 will be the 2 in 312 or the other 1 in 211. Both of these must be inserted into a slot at index larger than $j$ in $c$ which is not the first slot, so the evolution of $\pi$ must insert into a slot that is not the first slot.
\end{proof}

We have the following conditions for the set of vertical insertion encodings of Cayley permutations in $\Av(211, 312)$:
\begin{itemize}
      \item 
            As usual, when inserting the first value we must insert a new maximum value.
      \item 
            After an $\ell$ insertion, the next slot will be on the right of the last value that was inserted, so may be a new maximum or a repeated value.
      \item
            After an $r$ insertion, the next slot will be on the left of the last value that was inserted, so must be a new maximum.
      \item
            After an $m$ insertion, of the two slots we have inserted the first must begin with a new maximum and the second may be a new maximum or a repeated value.
\end{itemize}
This gives us the following non-ambiguous context free grammar for the set of vertical insertion encodings of $\Av(211, 312)$. The start symbol is $S$ which represents inserting a new maximum and $C$ represents inserting either a new maximum or a repeated value.

\begin{equation*}
      \begin{split}
            S & \to f_{1,1} \: |\: \ell_{1,1}C \:|\: r_{1,1} S \:|\: m_{1,1} SC         \\
            C & \to S \:|\: f_{1,0} \:|\: \ell_{1,0}C \:|\: r_{1,0} S \:|\: m_{1,0} SC
      \end{split}
\end{equation*}

Let $S(x)$ be the generating function for $S$ and $C(x)$ be the generating function for $C$. Then we have 

\begin{equation*}
      \begin{split}
            S(x) & = x + xC(x) + xS(x) + xS(x)C(x),         \\
            C(x) & = S(x) + x + xC(x) + xS(x) + xS(x)C(x),
      \end{split}
\end{equation*}

which solves to give the generating function
\begin{equation*}
      \begin{split}
            S(x) & = \frac{1-3x-\sqrt{1-6x+x^2}}{4x}
      \end{split}
\end{equation*}
for the sequence beginning $1, 3, 11, 45, 197, 903, 4279, 20793, 103049, 518859$ which is the little Schroeder numbers, A001003 in the OEIS~\cite{oeis}.

\section{Classifying regular languages}
\label{sec:regular-theory}
We will now focus on regular languages. 
The set of insertion encodings of a Cayley permutation class forms a regular language if it can be represented by a deterministic finite automaton for which there exists some $k \in \mathbb{N}_1$ such that from any state there exists a path of size $\leq k$ to an accept state.
Representing states as configurations and accept states as Cayley permutations,
from any configuration with $s$ slots the minimum number of insertions needed to obtain a Cayley permutation is $s$. Hence, for any class with a set of vertical insertion encodings which form a regular language, there is a bound on the number of slots in any configuration in the evolution of any Cayley permutation in the class. We will call such a class \emph{slot-bounded}.

In Section~\ref{sec:regular-algo}, we will show for finitely-based Cayley permutation classes that being slot-bounded is also a sufficient condition for it's vertical insertion encodings to form a regular language. We will do this by presenting an algorithm for finding a DFA for the language. In this section, we focus on giving a characterisation of slot-bounded Cayley permutation classes. 

Let $\SB(k)$ be the set of Cayley permutations with evolutions that contain configurations with at most $k$ slots. For each $k$, this is a finitely based Cayley permutation class. The following proposition allows us to compute the minimal Cayley permutations not in the class.

\begin{proposition}
    \label{prop: basis SB(k)}
    For each positive integer $k$, the set $\SB(k)$ is the set of Cayley permutations that avoid all size $2k + 1$ derivations of configurations of the form $~\diamond~ a_1 ~\diamond~ a_2 ~\diamond~ \cdots ~\diamond~ a_k ~\diamond~$, where $a_1a_2 \cdots a_k$ is any size $k$ Cayley permutation.
\end{proposition}
\begin{proof}
    Clearly, any Cayley permutation that is a derivation from such a configuration is not in $\SB(k)$. 
    
    Take a Cayley permutation $\pi$ not in $\SB(k)$. The evolution of $\pi$ must have a configuration with exactly $k + 1$ slots in its evolution as each insertion increases the number of slots by at most $1$. First, pick a number between each of the $k + 1$ slots, say $a_1, \ldots, a_k$. Second, pick a value of $\pi$ to fill each slot so that we have $\sigma_1a_1\sigma_2a_2...a_k\sigma_{K+1}$. These $2k + 1$ numbers will create a Cayley permutation of size $2k + 1$ that is a derivation of the configuration $~\diamond~ a_1 ~\diamond~ a_2 ~\diamond~ \cdots ~\diamond~ a_k ~\diamond~$.
\end{proof}

For $k=1$, the class $\SB(1)$ contains Cayley permutations whose evolutions have at most 1 slot. By Proposition \ref{prop: basis SB(k)}, the basis of $\SB(1)$ is all size 3 derivations of the configuration $~\diamond~ 1 ~\diamond~$, so $\SB(1) = \Av(112, 212, 213, 312)$. 
This class can also be described as the set of all Cayley permutations whose vertical insertion encodings do not use letters of the form $m_{i, j}$ for any $i$ and $j$.
If the slot in a configuration can have a repeated value inserted into it then we will describe that state as $A$, otherwise we will describe the state as $S$. Using this notation and beginning at $S$, we give the following regular grammar for the set of vertical insertion encodings of $\Av(112, 212, 213, 312)$.
\begin{equation*}
    \begin{split}
        S & \to f_{1,1} \: |\: \ell_{1,1}A \:|\: r_{1,1} S         \\
        A & \to S \:|\: f_{1,0} \:|\: \ell_{1,0}A \:|\: r_{1,0} S
    \end{split}
\end{equation*}
Let $S(x)$ be the generating function for $S$ and $A(x)$ be the generating function for $A$. Then
\begin{equation*}
    \begin{split}
        S(x) & = x + xA(x) + xS(x),        \\
        A(x) & = S(x) + x + xA(x) + xS(x), \\
             & = 2S(x),
    \end{split}
\end{equation*}
which solves to give the generating function

\begin{equation*}
    \begin{split}
        S(x) & = \frac{x}{1-3x}
    \end{split}
\end{equation*}
for the sequence beginning $1, 3, 9, 27, 81, 243$ which is A000244 in the OEIS~\cite{oeis}.

For a Cayley permutation class to be slot-bounded it must be a subclass of $\SB(k)$ for some $k$. In order to determine when this is the case, we need a few definitions.
A Cayley permutation $\pi$ with maximum value $n+k$ is a \emph{vertical juxtaposition} of a Cayley permutation $\sigma$ with maximum value $n$ and a Cayley permutation $\tau$ with maximum value $k$ if all entries in $\pi$ with values $n$ or smaller standardise to $\sigma$ and all entries in $\pi$ with values larger than $n$ standardise to $\tau$. For example, a vertical juxtaposition of $12432$ and $221$ is $16624352$ as the entries with value four and smaller are $12432$ and the entries with value larger than four are $665$.

There are nine classes of vertical juxtapositions which are central to our theory. They arise from vertically juxtaposing strictly increasing, strictly decreasing or constant sequences. We will use $\mathcal{V}_{a,b}$ with $a,b \in \{D, I, C \}$ to denote the Cayley permutation class which is the vertical juxtaposition of a strictly increasing sequence if $a = I$, strictly decreasing sequence if $a = D$ and constant sequence if $a = C$ with a strictly increasing sequence if $b = I$, strictly decreasing sequence if $b = D$ and constant sequence if $b = C$. For example, the Cayley permutation $23141$ belongs to the class $\mathcal{V}_{C, I}$ and $21122$ to $\mathcal{V}_{C,C}$, as shown in Figure \ref{fig: examples of vertical jux}. Each of $I$, $D$ and $C$ can be empty so $111 \in \mathcal{V}_{C,C}$ for example.

\begin{figure}[h]
    \begin{center}
        \resizebox{10cm}{!}{
            \begin{tikzpicture}
                \draw[step=1cm,gray,thin] (0.5,0.5) grid (5.5,4.5);
                \foreach \x/\y in {1/2, 2/3, 3/1, 4/4, 5/1}
                \fill[black] (\x,\y) circle (0.14cm);
                \draw[thick, dashed] (0.5,1.5) -- (5.5,1.5);
                \node[font = \Large] at (3, 0) {$23141 \in \mathcal{V}_{C, I}$  is a vertical };
                \node[font = \Large] at (3, -0.7) {juxtaposition of 11 and 123};
            \end{tikzpicture}
            \phantom{----}
            \begin{tikzpicture}
                \draw[step=1cm,gray, thin] (0.5,0.5) grid (5.5,2.5);
                \foreach \x/\y in {1/2, 2/1, 3/1, 4/2, 5/2}
                \fill[black] (\x,\y) circle (0.14cm);
                \draw[thick, dashed] (0.5,1.5) -- (5.5,1.5);
                \node[font = \Large] at (3, 0) {$21122 \in \mathcal{V}_{C, C}$  is a vertical};
                \node[font = \Large] at (3, -0.7) { juxtaposition of 11 and 111};
            \end{tikzpicture}
        }
    \end{center}
    \caption{Examples of vertical juxtapositions.}
    \label{fig: examples of vertical jux}
\end{figure}

A Cayley permutation can be in more than one of these nine classes, such as $23145$. The sequence $2345$ is strictly increasing, and the sequence $1$ is strictly increasing, strictly decreasing, and constant, hence $23145$ is in $\mathcal{V}_{I, I}$, $\mathcal{V}_{D, I}$ and $\mathcal{V}_{C, I}$. This is always the case when the values can be split so that there is a single $1$ on its own, hence the Cayley permutations $12$ and $21$ are in all nine of these classes of vertical juxtapositions.
Atkinson \cite[Theorem 2.2]{Atkinson1999} showed how to compute the basis for the juxtaposition of two permutation classes, a method which can be adapted to Cayley permutations. Table \ref{tab: vertical juxtapositions bases} shows these bases for Cayley permutations.

\begin{table}[h]
    \begin{tabular}{c|ccc}
        $A \backslash B$ & $I$                            & $D$                            & $C$                            \\
        \hline
        $I$              & $\Av(11,321,2143,2413)$        & $\Av(11,132,312)$              & $\Av(122,132,212,221,312,321)$ \\
        $D$              & $\Av(11,213,231)$              & $\Av(11,123,3142,3412)$        & $\Av(122,123,212,213,221,231)$ \\
        $C$              & $\Av(112,121,211,213,231,321)$ & $\Av(112,121,123,132,211,312)$ & $\Av(123,132,213,231,312,321)$ \\
    \end{tabular}
    \caption{The bases for the nine classes of vertical juxtapositions $\mathcal{V}_{A,B}$.}
    \label{tab: vertical juxtapositions bases}
\end{table}

A \emph{vertical alternation} is a Cayley permutation $a_1b_1a_2b_2 \cdots a_nb_n$ with $b_i > a_j$ for all $1 \leq i \leq n$ and $1 \leq j \leq n$. There is a unique vertical alternation of size $2n$ for each of the nine classes of vertical juxtapositions that we introduced. These have the following useful property. 

\begin{lemma}
    \label{lem: alternation implies juxtaposition}
    For each $\mathcal{V} \in \{ \mathcal{V}_{I,I}, \mathcal{V}_{I,D}, \mathcal{V}_{I,C}, \mathcal{V}_{D,I}, \mathcal{V}_{D,D}, \mathcal{V}_{D,C}, \mathcal{V}_{C,I}, \mathcal{V}_{C,D}, \mathcal{V}_{C,C} \}$, the vertical alternation of size $2n$ in $\mathcal{V}$ contains every vertical juxtaposition in $\mathcal{V}$ of size up to $n$.
\end{lemma}
\begin{proof}
    Let $a_1b_1a_2b_2...a_nb_n$ be the vertical alternation of size $2n$ in $\mathcal{V}$. Take an arbitrary $\pi$ of size $n$ in $\mathcal{V}$ that is a vertical juxtaposition of $\sigma$ and $\tau$. If the $i^{th}$ value of $\pi$ is from $\sigma$ select $a_i$ or if it is from $\tau$ select $b_i$. This selection of $n$ numbers will standardise to $\pi$ as the $b_i$'s are all greater than the $a_i$'s. Therefore, the vertical alternation of size $2n$ in $\mathcal{V}$ contains every vertical juxtaposition in $\mathcal{V}$ of size up to $n$.
\end{proof}

Using Lemma \ref{lem: alternation implies juxtaposition} we will show in Theorem \ref{thm: no vertical juxtapositions iff subclass of SB(k)} that to find when a Cayley permutation class is slot-bounded we need only to find out if its basis contains a vertical juxtaposition from each of our nine classes. In order to prove this, Lemma \ref{lem: subsequence of monotonic sequence} and Lemma \ref{lem: sequence n^4 implies strict sequence n} are needed first.

\begin{lemma}
    \label{lem: subsequence of monotonic sequence}
    A nondecreasing sequence of size $m$ contains a subsequence of size at least $\ceil{\sqrt{m}}$ that is either strictly increasing or constant.
\end{lemma}
\begin{proof}
    First, consider the case when $m=n^2$ for some $n\in\mathbb{N}_1$.
    If there are at least $n$ distinct values then there would be a strictly increasing sequence of size at least $n$.
    If there were fewer than $n$ distinct values then there would be at least one value with at least $n$ occurrences in the sequence, forming a constant subsequence of size at least $n$.
    
    If $n^2<m<(n+1)^2$ then either there are more than $n$ distinct values, in which case there is a strictly increasing subsequence of size at least $n + 1$, or there is least one value with at least $n + 1$ occurrences, in which case there is a constant subsequence of size at least $n + 1$. Hence, a strictly increasing sequence of size $m$, with $n^2 < m < (n+1)^2$ contains either a strictly increasing sequence or a constant sequence of size at least $\ceil{\sqrt{m}}$.
\end{proof}

Lemma~\ref{lem: subsequence of monotonic sequence} is a generalisation of the Erd\H{o}s-Szekeres theorem~\cite{Erdos1935} which states that for $r, s \in \mathbb{N}_1$ any sequence of length $(r-1)(s-1)+1$ contains a subsequence of length $r$ that is nondecreasing or a subsequence of length $s$ that is nonincreasing.
Using this theorem we can prove Lemma~\ref{lem: sequence n^4 implies strict sequence n}.

\begin{lemma}
    \label{lem: sequence n^4 implies strict sequence n}
    Any sequence of size $n^4$ contains a subsequence of size $n$ that is either strictly increasing, strictly decreasing or constant.
\end{lemma}
\begin{proof}
    By the Erd\H{o}s-Szekeres theorem~\cite{Erdos1935}, any sequence of size $n^4$ contains a subsequence of size $n^2$ that is nondecreasing or nonincreasing, and by Lemma \ref{lem: subsequence of monotonic sequence} any nondecreasing (or nonincreasing) sequence of size $n^2$ contains a subsequence of size $n$ that is either strictly increasing (or strictly decreasing) or constant.
\end{proof}

Using these lemmas, we generalise a proof on permutations by Albert, Linton and Ru\v{s}kuc \cite[Proposition 13]{Albert2005} to prove a condition in Theorem \ref{thm: no vertical juxtapositions iff subclass of SB(k)} for a Cayley permutation class to be a subclass of $\SB(k)$. 
As the results in Section~\ref{sec:regular-algo} show that a finitely based Cayley permutation class is regular if and only if it is slot-bounded, Theorem \ref{thm: no vertical juxtapositions iff subclass of SB(k)} gives a linear time algorithm in the length of the basis elements for checking if a finitely based Cayley permuation class is regular.

\begin{theorem}
    \label{thm: no vertical juxtapositions iff subclass of SB(k)}
    A Cayley permutation class $\Av(B)$ is a subclass of $\SB(k)$ for some $k$ if and only if $B$ contains a Cayley permutation from each of the nine classes $\mathcal{V}_{I, I}$, $\mathcal{V}_{I, D}$, $\mathcal{V}_{I, C}$, $\mathcal{V}_{D, I}$, $\mathcal{V}_{D, D}$, $\mathcal{V}_{D, C}$, $\mathcal{V}_{C, I}$, $\mathcal{V}_{C, D}$, and $\mathcal{V}_{C, C}$ of vertical juxtapositions.
\end{theorem}
\begin{proof} 
    Suppose that $\Av(B) \subseteq$ $\SB(k)$. 
    The nine Cayley permutations of the form $b_0a_1b_1...a_kb_k$ where $a_1...a_k$ and $b_0b_1...b_k$ are each strictly increasing, strictly decreasing or constant are examples from each of the nine types of vertical juxtapositions. They are also not in $\Av(B)$ as they are all derivations of configurations of the form $~\diamond~ a_1 ~\diamond~ a_2 ~\diamond~ \cdots ~\diamond~ a_k ~\diamond~$. As each of these nine forms are downwards closed, $B$ must contain a Cayley permutation from each of the nine classes of vertical juxtapositions.
    
    Suppose $\Av(B)$ is a Cayley permutation class with $B$ containing Cayley permutations from each of the nine classes of vertical juxtapositions. Suppose the longest vertical juxtaposition in $B$ has size $n$. We will show that there exists some integer $k$ such that every Cayley permutation in the basis of $\SB(k)$ contains a vertical juxtaposition that is contained in $B$, so $\Av(B) \subseteq \SB(k)$ for this $k$. 
    
    Let $k=n^{16}$ and $\sigma = b_1a_1b_2a_2b_3...b_ka_kb_{k+1}$ be a basis element of $\SB(k)$. Although a smaller $k$ may also work, we will show that $k = n^{16}$ is sufficient. 
    The sequence of $a$'s has size $k = n^{16}$, so by Lemma \ref{lem: sequence n^4 implies strict sequence n} there exists a subsequence of size $n^4$ which is either strictly increasing, strictly decreasing or constant. We will denote this as $a_{i_1} a_{i_2} \ldots a_{i_{n^4}}$.
    
    Next, we choose the $n^4$ values from the $b$'s that are adjacent to the $a_i$'s so that they form the pattern $a_{i_1} b_{i_1+1} a_{i_2} b_{i_2+1}\ldots a_{i_{n^4}} b_{i_{n^4+1}}$. As before, among these $b_i$'s there exists a subsequence of size $n$ that is strictly increasing, strictly decreasing, or constant by Lemma \ref{lem: sequence n^4 implies strict sequence n}. 
    We will denote these $b_{i_{j_1}} b_{i_{j_2}} \ldots b_{i_{j_n}}$.
    Taking a subset of the $a_i$'s that remain adjacent to the $b_{i_j}$'s we create the pattern $a_{i_{j_1}-1} b_{i_{j_1}} a_{i_{j_2}-1} b_{i_{j_2}} \ldots a_{i_{j_n}-1} b_{i_{j_n}}$.
    
    This sequence is a vertical alternation in one of the nine classes of vertical juxtapositions of length $2n$, so by Lemma \ref{lem: alternation implies juxtaposition} the sequence contains every vertical juxtaposition from the same class of size up to $n$, specifically a vertical juxtaposition that is in $B$. Therefore, this basis element $\sigma$ of $\SB(k)$ with $k = n^{16}$ contains a basis element from $B$. As $\sigma$ was arbitrary, every basis element of $\SB(k)$ contains an element from $B$, so none of them are in $\Av(B)$, hence $\Av(B) \subseteq$ $\SB(k)$.
\end{proof}

In Section~\ref{sec:regular-algo} we will give an algorithm for computing a DFA that accepts the regular language for the set of vertical insertion encodings of any finitely based slot-bounded Cayley permutation class. Hence, if a finitely based Cayley permutation class avoids arbitrarily long vertical juxtapositions then we can compute a DFA for it.

\section{Computing vertical insertion encodings}
\label{sec:regular-algo}
Vatter \cite{Vatter2009} gave a method for constructing a DFA for the insertion encodings of slot-bounded permutation classes. The method from that paper can be extended to slot-bounded Cayley permutation classes. However, we present an alternative approach that is more efficient for permutations and can be applied to Cayley permutations. Our definitions and methods follow those introduced by Albert, Bean, Claesson, Nadeau, Pantone, and Ulfarsson \cite{combexp} for permutations, but here we will adapt them to Cayley permutations. 

The heart of our approach is essentially the same as Vatter's method: consider an infinite accepting automata whose states are configurations with the start state $~\diamond~$, accept states being Cayley permutations, and transitions being determined by the possible insertions. Vatter outlines a method for determining when a letter can be removed from a configuration without changing the transitions, and calls such a letter \emph{insertion-encoding-reducible} (or \emph{IE-reducible}). From their results, in a slot-bounded permutation class any sufficiently long configuration will have an IE-reducible letter, which in turns tells us that this method will result in a DFA.

To implement this, one needs to verify that any permutation that is a derivation of the configuration that is not in the permutation class is still not in the permutation class after removing the letter. This is often a computationally expensive check.
Vatter showed that it is possible to determine if a letter in a configuration is IE-reducible by checking finitely many derivations of the configuration.

We will instead represent configurations as \emph{tilings}. Tilings give a convenient and computer implementable way of representing sets of Cayley permutations and the strategies needed to construct the DFA. 
To check if a letter is IE-reducible, Vatter's method requires the generation of permutations up to a fixed size with some checks on whether the permutations contain or avoid the basis of the class. 
Tilings keep track of the occurrences of basis elements as the configurations are expanded, so checks required to determine if a letter is IE-reducible are already done. In fact, by moving to tilings every letter will be IE-reducible, and therefore more states in our DFA are shown to be equivalent. For these reasons, our approach appears to produce significantly fewer states and require less computation time than Vatter's method, which we observed in the computational experiments that we ran. 

In Table~\ref{tab:states}, we compare the number of states in the DFAs found by our implementation of Vatter's method and the tilings approach for the classes $\Av(B)$ for subsets $B$ of the Cayley permutations of size $3$. It appears to show that the tilings approach results in a DFA with significantly fewer states.

\begin{table}[h]
    \centering
    \begin{tabular}{|c|ccc|ccc|}
        \hline
              & \multicolumn{3}{c|}{Vatter's method} & \multicolumn{3}{c|}{The tilings approach}                             \\
        \hline
        $|B|$ & Min                                  & Average                                   & Max & Min & Average & Max \\
        \hline
        3     & 33                                   & 180.7                                     & 843 & 5   & 17.9    & 38  \\
        4     & 12                                   & 124.1                                     & 616 & 3   & 15.6    & 35  \\
        5     & 9                                    & 92.2                                      & 434 & 4   & 13.7    & 28  \\
        6     & 13                                   & 70.7                                      & 330 & 4   & 12.1    & 23  \\
        \hline
    \end{tabular}
    \caption{Comparison of the number of states of the DFA found by Vatter's method and the tilings approach. This is for the classes $\Av(B)$ for subsets $B$ of the Cayley permutations of size $3$.}
    \label{tab:states}
\end{table}

We will now outline some of the details of this implementation in the remainder of this section. Much of this is a direct adaptation of the work of Albert \textit{et al.} \cite{combexp} to the context of Cayley permutations.
The material in Section 5.1 is technical with details that are independent of the rest of the paper so the reader may skip to Section 5.2 on their first pass through this paper if they wish.

\subsection{Tilings}
\label{sec:tilings}

A Cayley permutation $\pi$ can be represented on a grid such that each index of $\pi$ is given a position in a cell on the grid. For example, the Cayley permutation 35125224 is shown in Figure~\ref{fig:gridded cperm} with positions $((0,1), (0,1), (0,0), (0,0), (1,1), (1,0), (1,0), (1,1))$.

\begin{figure}[h]
    \centering
    \resizebox{0.2\textwidth}{!}{%
        \begin{tikzpicture}
            \draw[step=3cm,gray,very thin] (0,0) grid (6,6);
            \foreach \x/\y in {0.5/3.5, 1.125/5, 1.75/1, 2.375/2, 3.5/5, 4/2, 4.75/2, 5.5/3.5}
            \fill[black] (\x,\y) circle (0.15cm);
        \end{tikzpicture}
    }%
    \caption{The Cayley permutation $35125224$ on a grid with positions $((0,1), (0,1), (0,0), (0,0), (1,1), (1,0), (1,0), (1,1))$.}
    \label{fig:gridded cperm}
\end{figure}

More formally, a Cayley permutation $\pi$ of size $n$ is given an $n$-tuple of cells $(c_1, c_2, \ldots, c_n)$ with $c_i \in \mathbb{N}_0 \times \mathbb{N}_0$, called the $\emph{positions}$. This represents a gridding of the points of the Cayley permutation on the positive quadrant; the point $(i, \pi(i))$ with $c_i = (x, y)$ is in the square $[x, x+1) \times [y, y+1)$. 

We only consider tuples of positions for Cayley permutations which are consistent with the Cayley permutation, i.e., a Cayley permutation ${\pi(1)\pi(2)\cdots\pi(n)}$ and positions ${((x_1, y_1), (x_2, y_2), \ldots, (x_n, y_n))}$ are consistent if for every pair of indices $i$ and $j$ with $i < j$, we have
\begin{itemize}
    \item $x_i \leq x_j$,
    \item if $\pi(i) < \pi(j)$ then $y_i \leq y_j$,
    \item if $\pi(i) > \pi(j)$ then $y_i \geq y_j$, and
    \item if $\pi(i) = \pi(j)$ then $y_i = y_j$.
\end{itemize}
The example shown in Figure~\ref{fig:gridded cperm} is consistent but the Cayley permutation $212$ with positions $((0, 1), (1, 0), (1,0))$ is not consistent, for example.

We define a \emph{gridded Cayley permutation} to be a pair $(\pi, P)$ of a Cayley permutation $\pi$, called the \emph{underlying pattern}, with a tuple of positions $P$ such that $\pi$ and $P$ are consistent.
For example, $$(35125224, ((0,1), (0,1), (0,0), (0,0), (1,1), (1,0), (1,0), (1,1)))$$ is a gridded Cayley permutation shown in Figure \ref{fig:gridded cperm}, but $(212, ((0, 1), (1, 0), (1,0)))$ is not consistent, so is not a gridded Cayley permutation. 
For clarity when there are multiple gridded Cayley permutations on the same grid, we will represent each gridded Cayley permutation with lines joining it's points.
We define the \emph{size} of a gridded Cayley permutation $(\pi, P)$ to be the size of its underlying pattern $\pi$. As a shorthand, when all of the cells in $P$ are the same, say $c$, then we will write $(\pi, c)$.

Let $\mathcal{GC}$ be the set of all gridded Cayley permutations, and $\mathcal{GC}^{(t, u)}$ be the set of gridded Cayley permutations with points in the region $[0, t] \times [0, u]$. A gridded Cayley permutation $(\pi, (c_1, \ldots, c_n))$ \emph{contains} a gridded Cayley permutation $(\sigma, (d_1, \ldots, d_k))$ if there is a subsequence $\pi(i_1)\pi(i_2)\cdots\pi(i_k)$ of $\pi$ that standardises to $\sigma$, and the position $c_{i_j} = d_j$ for all $1 \leq j \leq k$.
As for sets of Cayley permutations, we say that a gridded Cayley permutation avoids a set of gridded Cayley permutations if it avoids every gridded Cayley permutation in the set. Otherwise, we say the gridded Cayley permutation contains the set of gridded Cayley permutations.
In Figure~\ref{fig:gridded cperm avoidance}, the points circled in red show an occurrence of the gridded Cayley permutation $(2311, ((0,1), (0,1), (1,0), (1,0)))$ in  $(35125224, ((0,1), (0,1), (0,0), (0,0), (1,1), (1,0), (1,0), (1,1)))$.

\begin{figure}[h]
    \centering
    \resizebox{0.2\textwidth}{!}{%
        \begin{tikzpicture}
            \draw[step=3cm,gray,very thin] (0,0) grid (6,6);
            \foreach \x/\y in {0.5/3.5, 1.125/5, 1.75/1, 2.375/2, 3.5/5, 4/2, 4.75/2, 5.5/3.5}
            \fill[black] (\x,\y) circle (0.15cm);
            
            \foreach \x/\y in {0.5/3.5, 1.125/5, 4/2, 4.75/2}
            \draw[color=red] (\x,\y) circle (0.25cm);
        \end{tikzpicture}
    }%
    \caption{An occurrence of the gridded Cayley permutation $(2311, ((0,1), (0,1), (1,0), (1,0)))$ in  $(35125224, ((0,1), (0,1), (0,0), (0,0), (1,1), (1,0), (1,0), (1,1)))$.}
    \label{fig:gridded cperm avoidance}
\end{figure}

A \emph{(Cayley) tiling} is a triple $\mathcal{T} = ((t, u), \mathcal{O}, \mathcal{R})$ where $(t, u) \in \mathbb{N}_0 \times \mathbb{N}_0$, $\mathcal{O} \subseteq \mathcal{GC}^{(t, u)}$ is a set of gridded Cayley permutations called \emph{obstructions}, and $\mathcal{R} = \{\mathcal{R}_1, \ldots, \mathcal{R}_k\}$ is a set of sets of gridded Cayley permutations called \emph{requirements}. Each $\mathcal{R}_i$ is called a \emph{requirement list}. The tiling $\mathcal{T}$ represents the set $\Grid(\mathcal{T})$ of gridded Cayley permutations in $\mathcal{GC}^{(t, u)}$ that avoid $\mathcal{O}$ and contain each $\mathcal{R}_i$.

For a Cayley permutation class $\Av(B)$, the set of derivations from the configuration $~\diamond~$ is the same as the set of non-empty Cayley permutations in $\Av(B)$. There exists a size-preserving bijection between the set $\Av(B)$ and the gridded Cayley permutations represented by the $1 \times 1$ tiling that avoids each pattern in $B$ in the cell $(0,0)$ while containing a point in cell $(0, 0)$, as stated in Proposition~\ref{prop: tilings bijection to slot}. 

\begin{proposition}
    \label{prop: tilings bijection to slot}
    For a Cayley permutation class $\Av(B)$, the set of derivations $D_{\Av(B)}(~\diamond~)$ is in bijection with the gridded Cayley permutations represented by the tiling $$((1, 1), \{(\pi, (0, 0)) \mid \pi \in B\}, \{\{(1, (0, 0))\}\}).$$
\end{proposition}

We usually represent obstructions in red with circles and requirements in blue with squares, as in Figure~\ref{fig: single slot tiling Av(231, 321)} which shows the tiling representing the set of derivations from the configuration $~\diamond~$ in the class $\Av(11, 231, 321)$. Every basis element is represented as an obstruction in the cell $(0,0)$ and the requirement represents a slot.

\begin{figure}[h]
    \centering
    \begin{tikzpicture}
        \node (simple) at (5,1) {
        \tiling{2.0}{1}{1}{}%
        {%
        {3/{(0.45, 0.8), (0.65, 0.6), (0.85, 0.4)}}, {3/{(0.2, 0.25), (0.35, 0.4), (0.5, 0.1)}}, {2/{(0.6, 0.25), (0.8, 0.25)}}}
        {%
        {1/{(0.5, 0.5)}}%
        }};
        
    \end{tikzpicture}
    \caption{The tiling representing the set of derivations from the configuration $~\diamond~$ in the class $\Av(11, 231, 321)$.}
    \label{fig: single slot tiling Av(231, 321)}
\end{figure}

Every configuration can also be represented as a tiling.
We effectively "draw" a configuration on a grid where each placed point in the configuration is represented by a cell in the tiling, point requirements are used to represent slots, and all other cells contain point obstructions to ensure that the order of points being placed in the tiling is consistent with the vertical insertion encoding. This creates an \emph{intermediate tiling} representing the configuration. On top of this we add obstructions representing the basis of the class to ensure that the gridded Cayley permutations represented by the tiling are exactly the derivations from the configuration that are in the class.

For a configuration with $n$ points already placed with $m$ different values and $k$ slots, we create a size $(n+k) \times (m+1)$ tiling. 
Each placed point with value $v$ at index $i$ in the configuration is represented by a \emph{point cell} in the cell $(i, v)$. This consists of the following.
\begin{itemize}
    \item A requirement list containing the point requirement $\{(1, (i, v))\}$.
    \item Point obstructions in every other cell in column $i$ to ensure that there are no other values in that column. That is, for every $u \in [m+1]$ with $u \neq v$, add the obstruction $(1, (i, u))$.
    \item The obstructions with underlying Cayley permutations $12$ and $21$ and positions in cells $(x, v)$ for all $x \in [n+k]$ to ensure that there can be no other values in that row.
    \item The obstruction $(11, (i, v))$ to ensure that there is exactly one point in the point cell.
    \item For each cell $(x, y)$ such that $x < i$ and $y = v$, if $(x, y)$ is not also a point cell then add the obstruction $(1, (x, y))$. This ensures that no occurrences of the same value can appear to the left of the placed point, following the rules of vertical insertion encoding.
\end{itemize}
Each slot at index $i$ in the configuration is represented by a list requirement $\mathcal{R}$ containing the requirements $(1, (i, y))$ for each $y \in [m+1]$ such that $(i, y)$ does not contain a point obstruction. If $i$ is to the left of the largest, rightmost point which has been placed in the configuration, then this will be a size 1 list as the only option for this slot is to insert a new maximum value. Otherwise, it will be a size 2 list, representing the options of inserting a new maximum value or a repeated value.
The intermediate tiling for the configuration $~2~\diamond~2~1~\diamond~$ is shown in Figure~\ref{fig:config as tiling}.

\begin{figure}[h]
    \centering 
    \centering   
    \resizebox{0.35\textwidth}{!}{%
        \begin{tikzpicture}
            \node (example tiling) at (1,1) {
            \tiling{2.0}{5}{3}{}%
            {%
            {1/{(0.5, 0.5)}}, {1/{(1.5, 0.5)}}, {1/{(2.5, 0.5)}}, {1/{(4.5, 0.5)}}, {1/{(0.5, 2.5)}}, {1/{(1.5, 1.5)}}, {1/{(3.5, 1.5)}}, {1/{(2.5, 2.5)}}, {1/{(3.5, 2.5)}}, {2/{(0.2, 1.4), (0.4, 1.2)}}, {2/{(0.6, 1.2), (0.8, 1.4)}}, 2/{(0.35, 1.7), (0.65, 1.7)}, {2/{(2.2, 1.4), (2.4, 1.2)}}, {2/{(2.6, 1.2), (2.8, 1.4)}}, 2/{(2.35, 1.65), (2.65, 1.65)}, {2/{(3.2, 0.4), (3.4, 0.2)}}, {2/{(3.6, 0.2), (3.8, 0.4)}}, 2/{(3.35, 0.7), (3.65, 0.7)},
            {2/{(0.8, 1.6), (4.2, 1.9)}}, {2/{(0.8, 1.9), (4.3, 1.6)}}, 2/{(0.8, 1.1), (2.2, 1.3)}, 2/{(0.8, 1.3), (2.2, 1.1)}, 2/{(2.8, 1.1), (4.2, 1.3)}, 2/{(2.8, 1.3), (4.2, 1.1)}, {2/{(4.2, 1.4), (4.4, 1.2)}}, {2/{(4.6, 1.2), (4.8, 1.4)}}%
            }%
            {{1/{(1.5,2.5)}}, {1/{(4.5,1.5)}}, {1/{(4.5,2.5)}}, {1/{(0.5,1.5)}}, {1/{(2.5,1.5)}}, {1/{(3.5,0.5)}}}};
        \end{tikzpicture}
    }%
    \caption{The intermediate tiling representing the configuration $~2~\diamond~2~1~\diamond~$.}
    \label{fig:config as tiling}
\end{figure}

Formally, this tiling is $\mathcal{T} = ((2, 5), \mathcal{O}, \{\mathcal{R}_1, \mathcal{R}_2, \mathcal{R}_3, \mathcal{R}_4, \mathcal{R}_5\})$ where
\begin{itemize}
    \item $\mathcal{O} = \{(1, (0, 0)),$ $(1, (1, 0)),$ $(1, (2, 0)),$ $(1, (4, 0)),$ $(1, (1, 1)),$ $(1, (3, 1)),$ $(1, (0, 2)),$ $(1, (2, 2)),$ $(1, (3, 2))$,
          $(11, (0, 1))$, $(12, (0, 1))$, $(21, (0, 1))$, 
          $(11, (2, 1))$, $(12, (2, 1))$, $(21, (2, 1))$,
          $(11, (3, 0))$, $(12, (3, 0))$, $(21, (3, 0))$,
          $(12, ((0, 1), (2, 1)))$, $(21, ((0, 1), (2, 1)))$,
          $(12, ((0, 1), (4, 1)))$, \\ $(21, ((0, 1), (4, 1)))$, 
          $(12, ((2, 1), (4, 1)))$, $(21, ((2, 1), (4, 1)))$,
          $(12, (4, 1))$, $(21, (4, 1))$, $\}$
    \item $\mathcal{R}_1 = \{(1, (0, 1))\}, $
    \item $\mathcal{R}_2 = \{(1, (2, 1))\}, $
    \item $\mathcal{R}_3 = \{(1, (3, 0))\}, $
    \item $\mathcal{R}_4 = \{(1, (1, 2))\}, $ and
    \item $\mathcal{R}_5 =  \{(1, (4, 1)), (1, (4, 2))\}$.
\end{itemize}

To reduce the notation in tilings, point cells will be represented by a black dot replacing the obstructions and requirement list in that cell, as shown in cells $(0, 1)$, $(2,1)$ and $(3,0)$ in Figure~\ref{fig:config as tiling, point cells}. Additionally, the size 2 obstructions in rows containing a point cell will be omitted for clarity and represented by a $*$ beside that row in the tiling, also shown in Figure~\ref{fig:config as tiling, point cells}. These rows are called \emph{point rows}. Note that apart from the top row of a tiling every other row will be a point row.

\begin{figure}[h]
    \centering 
    \centering   
    \begin{tikzpicture}
        \node (example tiling) at (1,1) {
        \tiling{1.0}{5}{3}{0/1, 2/1, 3/0}%
        {%
        {1/{(0.5, 0.5)}}, {1/{(1.5, 0.5)}}, {1/{(2.5, 0.5)}}, {1/{(4.5, 0.5)}}, {1/{(0.5, 2.5)}}, {1/{(1.5, 1.5)}}, {1/{(3.5, 1.5)}}, {1/{(2.5, 2.5)}}, {1/{(3.5, 2.5)}}%
        }
        {1/{(1.5,2.5)}, 1/{(4.5,1.5)}, 1/{(4.5,2.5)}}};
        \node at (3.7, 1) {$*$};
        \node at (3.7, 0) {$*$};
    \end{tikzpicture}
    \caption{Intermediate tiling representing the configuration $~2~\diamond~2~1~\diamond~$ with point cells in black and point rows denoted $*$.}
    \label{fig:config as tiling, point cells}
\end{figure}

From this intermediate tiling, for a class $\Av(B)$ we add 
all of the gridded Cayley permutations in $\mathcal{GC}^{(5, 3)}$ with an underlying Cayley permutation in $B$ as obstructions.
Albert \textit{et al.} \cite{combexp} outlined how effective keeping track of occurrences of basis patterns can be in the context of permutation classes. 
We also simplify tilings where possible to create equivalent tilings.
We say that two tilings $\mathcal{T}_1$ and $\mathcal{T}_2$ are \emph{equivalent}, denoted $\mathcal{T}_1 \cong \mathcal{T}_2$ if there exists a size preserving bijection between the gridded Cayley permutations in $\Grid(\mathcal{T}_1)$ and in $\Grid(\mathcal{T}_2)$. A tiling $\mathcal{T}_1$ can often be simplified to an equivalent tiling $\mathcal{T}_2$ with fewer obstructions and requirements by the following.
\begin{enumerate}
    \item If an obstruction contains another obstruction, remove the larger obstruction. Avoiding the smaller one is sufficient to avoid the larger one.
    \item If an obstruction has a position in a point cell, then that position and index in the underlying pattern can be removed from the obstruction.
    \item If a row or column in a tiling contains point obstructions in every cell, then that row or column can be removed from the tiling.
    \item If there exists a point obstruction $g = (1, (x, y))$ and there exists a requirement list $\mathcal{R}$ containing $g$, then $g$ is removed from $\mathcal{R}$. The obstruction implies that a point in that cell would create an occurrence of a basis pattern and a different requirement in the list must be satisfied. If there are no other requirements in $\mathcal{R}$, then the tiling represents a configuration with no derivations in the class.
\end{enumerate}
By the first simplification, for the class $\Av(123)$ we need only add obstructions with pattern $123$ and positions in cells which do not contain point obstructions. Further, obstructions can have at most one index in point cells as they contain all size 2 obstructions and at most one value in point rows. Figure~\ref{subfig: example tiling in class} shows the tiling after adding the basis pattern 123 in every possible way after this first simplification.

\begin{figure}[h]
    \centering
    \begin{subfigure}[c]{0.3\textwidth}
        \centering   
        \resizebox{0.95\textwidth}{!}{%
            \begin{tikzpicture}
                \node (example tiling) at (1,1) {
                \tiling{1.0}{5}{3}{0/1, 2/1, 3/0}%
                {
                {1/{(0.5, 0.5)}}, {1/{(1.5, 0.5)}}, {1/{(2.5, 0.5)}}, {1/{(4.5, 0.5)}}, {1/{(0.5, 2.5)}}, {1/{(1.5, 1.5)}}, {1/{(3.5, 1.5)}}, {1/{(2.5, 2.5)}}, {1/{(3.5, 2.5)}},
                {3/{(0.3, 1.6), (1.2, 2.4), (1.4, 2.6)}}, {3/{(0.8, 1.4), (4.2, 2.7), (4.4, 2.8)}}, {3/{(2.8, 1.8), (4.1, 2.5), (4.3, 2.6)}}, {3/{(3.3, 0.8), (4.1, 2.2), (4.23, 2.4)}}, {3/{(3.6, 0.7), (4.2, 1.7), (4.5, 2.2)}}, {3/{(4.6, 2.6), (4.75, 2.75), (4.9, 2.9)}}, {3/{(4.5, 1.75), (4.75, 2.25), (4.9, 2.53)}}%
                }
                {1/{(1.5,2.5)}, 1/{(4.5,1.5)}, 1/{(4.5,2.5)}}};
                \node at (3.7, 1) {$*$};
                \node at (3.7, 0) {$*$};
            \end{tikzpicture}
        }%
        \subcaption{Add the basis pattern 123 as an obstruction in every way.}
        \label{subfig: example tiling in class}
    \end{subfigure}
    \begin{subfigure}[c]{0.3\textwidth}
        \centering   
        \resizebox{0.95\textwidth}{!}{%
            \begin{tikzpicture}
                \node (example tiling) at (1,1) {
                \tiling{1.0}{5}{3}{0/1, 2/1, 3/0}%
                {
                {1/{(0.5, 0.5)}}, {1/{(1.5, 0.5)}}, {1/{(2.5, 0.5)}}, {1/{(4.5, 0.5)}}, {1/{(0.5, 2.5)}}, {1/{(1.5, 1.5)}}, {1/{(3.5, 1.5)}}, {1/{(2.5, 2.5)}}, {1/{(3.5, 2.5)}},
                {2/{(1.2, 2.4), (1.4, 2.6)}}, {2/{(4.2, 2.7), (4.4, 2.8)}}, {2/{(4.1, 2.5), (4.3, 2.6)}}, {2/{(4.1, 2.2), (4.23, 2.4)}}, {2/{(4.2, 1.7), (4.5, 2.2)}}, {3/{(4.6, 2.6), (4.75, 2.75), (4.9, 2.9)}}, {3/{(4.5, 1.75), (4.75, 2.25), (4.9, 2.53)}}%
                }
                {1/{(1.5,2.5)}, 1/{(4.5,1.5)}, 1/{(4.5,2.5)}}};
                \node at (3.7, 1) {$*$};
                \node at (3.7, 0) {$*$};
            \end{tikzpicture}
        }%
        \subcaption{Remove indices of basis patterns which are in point cells.}
        \label{subfig: example tiling in class, simplified once}
    \end{subfigure}
    \begin{subfigure}[c]{0.3\textwidth}
        \centering   
        \resizebox{0.95\textwidth}{!}{%
            \begin{tikzpicture}
                \node (example tiling) at (1,1) {
                \tiling{1.0}{5}{3}{0/1, 2/1, 3/0}%
                {
                {1/{(0.5, 0.5)}}, {1/{(1.5, 0.5)}}, {1/{(2.5, 0.5)}}, {1/{(4.5, 0.5)}}, {1/{(0.5, 2.5)}}, {1/{(1.5, 1.5)}}, {1/{(3.5, 1.5)}}, {1/{(2.5, 2.5)}}, {1/{(3.5, 2.5)}},
                {2/{(1.2, 2.4), (1.4, 2.6)}}, {2/{(4.2, 2.7), (4.4, 2.8)}},  {2/{(4.2, 1.7), (4.5, 2.2)}}%
                }
                {1/{(1.5,2.5)}, 1/{(4.5,1.5)}, 1/{(4.5,2.5)}}};
                \node at (3.7, 1) {$*$};
                \node at (3.7, 0) {$*$};
            \end{tikzpicture}
        }%
        \subcaption{Simplify obstructions contained by other obstructions.}
        \label{subfig: example tiling in class, simplified twice}
    \end{subfigure}
    \caption{Adding basis elements from the class $\Av(123)$ to an intermediate tiling as obstructions and simplifying.}
    \label{fig:simplifying a tiling}
\end{figure}

By the second simplification, any obstructions with an index in a point cell can be simplified to a smaller obstruction. This is done in Figure~\ref{subfig: example tiling in class, simplified once}. And finally, the obstructions can be simplified again by containment, as in Figure~\ref{subfig: example tiling in class, simplified twice}.
More details of simplifying tilings are given in Section 6.3.2 of Albert \textit{et al.} \cite{combexp}, with some additional care needed when considering  point rows. However, the above simplifications are sufficient for our purposes.

We will now explain how a Cayley permutation class can be enumerated using tilings, using the class $\Av(123)$ as an example.
Begin with the tiling representing the configuration $~\diamond~$ for the class, as shown in Figure~\ref{fig: single slot tiling Av(123)} for our example.

\begin{figure}[h]
    \centering
    \begin{tikzpicture}
        \node (simple) at (0,0) {
        \tiling{2.0}{1}{1}{}%
        {%
        {3/{(0.55, 0.8), (0.35, 0.6), (0.15, 0.4)}}}
        {%
        {1/{(0.5, 0.5)}}%
        }};
    \end{tikzpicture}
    \caption{The tiling representing the the configuration $~\diamond~$ in the class $\Av(123)$.}
    \label{fig: single slot tiling Av(123)}
\end{figure}

For each tiling, an insertion can be made into any cell containing a point requirement as these represent the slots in the configuration.
From $~\diamond~$, there are four possible insertions that can be made; $f_{1, 1}$, $\ell_{1, 1}$, $m_{1, 1}$, and $r_{1, 1}$. 
Their intermediate tilings are shown in Figure~\ref{fig: four intermediate tilings}.
Each intermediate tiling is size $3 \times 2$ with a point cell in position $(1, 0)$ representing the newly added point and all the obstructions and requirements necessary for a point cell.
For each type of insertion $\ell$, $m$, $r$, and $f$, we add 
the gridded Cayley permutations $g_1 = (1, (0,1))$, $g_2 = (1, (2,0))$, and $g_3 = (1, (2,1))$ as either obstructions or requirements.
For an $f$ insertion the slot has been filled so both the first and third column are empty, so all of $g_1$, $g_2$, and $g_3$ are added as obstructions. For an $l$ insertion, $g_1$ is added as an obstruction and $R_2 = \{g_2, g_3\}$ is added as a requirement list, as the first column must be empty but the third column must contain a point.
Similarly, for an $r$ insertion, $R_1 = \{g_1\}$ is added as a requirement list and $g_2$ and $g_3$ are added as obstructions, and for an $m$ insertion both $R_1$ and $R_2$ are added as requirement lists.

\begin{figure}[h]
    \centering
    \begin{tikzpicture}
        \node (f insertion) at (0,0) {
        \tiling{1.0}{3}{2}{1/0}%
        {%
        {1/{(2.5, 0.5)}}, {1/{(0.5, 1.5)}}, {1/{(2.5, 1.5)}}, {1/{(0.5, 0.5)}}, {1/{(1.5, 1.5)}}%
        }
        {}};
        \node (star) at (1.7,-0.5) {$*$};
        \node (f_1,1) at (0, -1.5) {$1$};
        
        \node (l insertion) at (4,0) {
        \tiling{1.0}{3}{2}{1/0}%
        {%
        {1/{(0.5, 1.5)}}, {1/{(2.5, 1.5)}}, {1/{(0.5, 0.5)}}, {1/{(1.5, 1.5)}}%
        }
        {1/{(2.5,1.5)}, 1/{(2.5,0.5)}}};
        \node (star) at (5.7,-0.5) {$*$};
        \node (l_1,1) at (4, -1.5) {$~1~\diamond~$};
        
        \node (r insertion) at (8,0) {
        \tiling{1.0}{3}{2}{1/0}%
        {%
        {1/{(2.5, 0.5)}}, {1/{(2.5, 1.5)}}, {1/{(0.5, 0.5)}}, {1/{(1.5, 1.5)}}%
        }
        {1/{(0.5,1.5)}}};
        \node (star) at (9.7,-0.5) {$*$};
        \node (r_1,1) at (8, -1.5) {$~\diamond~1~$};

        \node (m insertion) at (12,0) {
        \tiling{1.0}{3}{2}{1/0}%
        {%
        {1/{(0.5, 1.5)}}, {1/{(2.5, 1.5)}}, {1/{(0.5, 0.5)}}, {1/{(1.5, 1.5)}}%
        }
        {1/{(0.5,1.5)}, 1/{(2.5,1.5)}, 1/{(2.5,0.5)}}}; 
        \node (star) at (13.7,-0.5) {$*$};
        \node (m_1,1) at (12, -1.5) {$~\diamond~1\diamond~$};
    \end{tikzpicture}
    \caption{The four intermediate tilings representing configurations after one insertion into $~\diamond~$ with the configuration they represent underneath.}
    \label{fig: four intermediate tilings}
\end{figure}

For each of these intermediate tilings, we add all of the gridded Cayley permutations in $\mathcal{GC}^{(3, 2)}$ with an underlying Cayley permutation in $B = \{123\}$ as obstructions and simplify as before.
These four tilings are shown in Figure~\ref{fig: first rule for Av(123)} and represent all of the possible configurations after one insertion into the configuration $~\diamond~$ in the class $\Av(123)$. These tilings are called the \emph{children} of the original tiling, which is called the \emph{parent} tiling. We only consider a tiling to be a child of it's parent if it has derivations in the class (see simplification 4). Together they form a \emph{rule}, mapping the parent to its children which represent all possible configurations after one insertion. This is shown in Figure~\ref{fig: first rule for Av(123)} with a label for each tiling. 

\begin{figure}[h]
    \centering
    \begin{tikzpicture}
        \node (start) at (0,0) {
        \tiling{2.0}{1}{1}{}%
        {%
        {3/{(0.55, 0.8), (0.35, 0.6), (0.15, 0.4)}}}
        {%
        {1/{(0.5, 0.5)}}%
        }};
        \node (label S) at (0, -1.5) {$S$};
        
        \node (arrow) at (1.5,0) {$\to$};
        
        \node (f insertion) at (2.5,0) {
            \tiling{1.0}{1}{1}{0/0}%
            {%
            }
            {}};
        \node (star) at (3.2,0) {$*$};
        \node (label A) at (2.5, -1) {$A$};
        
        \node (plus) at (3.7,0) {$+$};
        
        \node (l insertion) at (5.2,0) {
        \tiling{1.0}{2}{2}{0/0}%
        {%
        {1/{(0.5, 1.5)}}, {1/{(1.5, 1.5)}}, {2/{(1.45, 1.8), (1.25, 1.6)}}%
        }
        {1/{(1.5,1.5)}, 1/{(1.5,0.5)}}};
        \node (star) at (6.4,-0.5) {$*$};
        \node (label B) at (5.2, -1.5) {$B$};
        
        \node (plus) at (6.9,0) {$+$};
        
        \node (r insertion) at (8.5,0) {
        \tiling{1.0}{2}{2}{1/0}%
        {%
        {1/{(0.5, 0.5)}}, {1/{(1.5, 1.5)}}, {3/{(0.55, 1.8), (0.35, 1.6), (0.15, 1.4)}}%
        }
        {1/{(0.5,1.5)}}};
        \node (star) at (9.7,-0.5) {$*$};
        \node (label D) at (8.5, -1.5) {$D$};
        
        \node (plus) at (10.2,0) {$+$};
        
        \node (m insertion) at (12.2,0) {
        \tiling{1.0}{3}{2}{1/0}%
        {%
        {1/{(0.5, 1.5)}}, {1/{(2.5, 1.5)}}, {1/{(0.5, 0.5)}}, {1/{(1.5, 1.5)}}, {3/{(0.55, 1.8), (0.35, 1.6), (0.15, 1.4)}}, {2/{(2.45, 1.8), (2.25, 1.6)}}%
        }
        {1/{(0.5,1.5)}, 1/{(2.5,1.5)}, 1/{(2.5,0.5)}}}; 
        \node (star) at (13.9,-0.5) {$*$};
        \node (label C) at (12.2, -1.5) {$C$};
    \end{tikzpicture}
    \caption{A rule showing the parent tiling representing $~\diamond~$ and its children representing all possible configurations after one insertion in the class $\Av(123)$.}
    \label{fig: first rule for Av(123)}
\end{figure}

We can keep going in this manner by inserting the leftmost minimum point into each new tiling which contains requirements. If our letter is of the form $x_{i, 0}$ then we place the point in a point row and if our letter is of the form $x_{i, 1}$ then we place the point in a row which is not a point row.
For example, the children of the tiling labelled $B$ in Figure~\ref{fig: first rule for Av(123)} after simplification are shown in Figure~\ref{fig: another rule for Av(123)}. The configuration $B$ represents, $~1~\diamond~$, has eight possible insertions; $f_{1, 0}$, $\ell_{1, 0}$, $r_{1, 0}$, $m_{1, 0}$, $f_{1, 1}$, $\ell_{1, 1}$, $r_{1, 1}$, and $m_{1, 1}$. The children of $B$ are shown in the same respective order.

\begin{figure}[h]
    \centering
    \resizebox{\linewidth}{!}{%
        \begin{tikzpicture}
            \node (start) at (0.7,0) {
            \tiling{1.0}{2}{2}{0/0}%
            {%
            {1/{(0.5, 1.5)}}, {1/{(1.5, 1.5)}}, {2/{(1.45, 1.8), (1.25, 1.6)}}%
            }
            {1/{(1.5,1.5)}, 1/{(1.5,0.5)}}};
            \node (label B) at (0.7, -1.5) {$B$};
            
            \node (arrow) at (2,0) {$\to$};
            
            \node (f insertion) at (3.4,0) {
                \tiling{1.0}{2}{1}{0/0, 1/0}%
                {%
                }
                {}};
            \node (star) at (4.6,0) {$*$};
            \node (label) at (3.4, -1) {$E$};
            
            \node (plus) at (4.9,0) {$+$};
            
            \node (l insertion) at (6.7,0) {
            \tiling{1.0}{3}{2}{0/0, 1/0}%
            {%
            {1/{(0.5, 1.5)}}, {1/{(1.5, 1.5)}}, {2/{(2.45, 1.8), (2.25, 1.6)}}%
            }
            {1/{(2.5,1.5)}, 1/{(2.5,0.5)}}};
            \node (star) at (8.4,-0.5) {$*$};
            \node (label) at (6.7, -1.5) {$F$};
            
            \node (plus) at (8.6,0) {$+$};
            
            \node (r insertion) at (10.35,0) {
            \tiling{1.0}{3}{2}{0/0,2/0}%
            {%
            {1/{(0.5, 1.5)}}, {1/{(1.5, 0.5)}}, {1/{(2.5, 1.5)}}, {2/{(1.45, 1.8), (1.25, 1.6)}}%
            }
            {1/{(1.5,1.5)}}};
            \node (star) at (12,-0.5) {$*$};
            \node (label) at (10.35, -1.5) {$G$};
            
            \node (plus) at (12.35,0) {$+$};
            
            \node (m insertion) at (14.7,0) {
            \tiling{1.0}{4}{2}{0/0,2/0}%
            {%
            {1/{(0.5, 1.5)}}, {1/{(1.5, 1.5)}}, {1/{(3.5, 1.5)}}, {1/{(1.5, 0.5)}}, {1/{(2.5, 1.5)}}, {2/{(1.45, 1.8), (1.25, 1.6)}}, {2/{(3.45, 1.8), (3.25, 1.6)}}%
            }
            {1/{(1.5,1.5)}, 1/{(3.5,1.5)}, 1/{(3.5,0.5)}}}; 
            \node (star) at (17,-0.5) {$*$};
            \node (label) at (14.7, -1.5) {$H$};
            
            
            \node (plus) at (2.5,-4) {$+$};
            
            \node (f insertion) at (3.9,-4) {
            \tiling{1.0}{2}{2}{0/0, 1/1}%
            {%
            {1/(0.5, 1.5)}, {1/{(1.5, 0.5)}}%
            }
            {}};
            \node (star) at (5.1,-4.5) {$*$};
            \node (star) at (5.1,-3.5) {$*$};
            \node (label) at (3.9, -5.5) {$I$};
            
            \node (plus) at (5.4,-4) {$+$};
            
            \node (l insertion) at (7.2,-4) {
            \tiling{1.0}{3}{2}{0/0, 1/1}%
            {%
            {1/{(0.5, 1.5)}}, {1/{(1.5, 0.5)}}, {1/{(2.5, 0.5)}}%
            }
            {1/{(2.5,1.5)}}};
            \node (star) at (8.9,-3.5) {$*$};
            \node (star) at (8.9,-4.5) {$*$};
            \node (label) at (7.2, -5.5) {$J$};
            
            \node (plus) at (9.1,-4) {$+$};
            
            \node (r insertion) at (10.85,-4) {
            \tiling{1.0}{3}{3}{0/0,2/1}%
            {%
            {1/{(1.5, 1.5)}}, {1/{(2.5, 0.5)}}, {1/{(0.5, 1.5)}}, {1/{(0.5, 2.5)}}, {1/{(1.5, 0.5)}}, {1/{(2.5, 2.5)}}, {2/{(1.45, 2.8), (1.25, 2.6)}}%
            }
            {1/{(1.5,2.5)}}};
            \node (star) at (12.5,-5) {$*$};
            \node (star) at (12.5,-4) {$*$};
            \node (label) at (10.85, -6) {$K$};
            
            \node (plus) at (12.85,-4) {$+$};
            
            \node (m insertion) at (15.2,-4) {
            \tiling{1.0}{4}{3}{0/0,2/1}%
            {%
            {1/{(0.5, 1.5)}}, {1/{(1.5, 0.5)}}, {1/{(2.5, 0.5)}}, {1/{(3.5, 0.5)}}, {1/{(0.5, 2.5)}}, {1/{(1.5, 2.5)}}, {1/{(3.5, 2.5)}}, {1/{(1.5, 1.5)}}, {1/{(2.5, 2.5)}}, {2/{(1.45, 2.8), (1.25, 2.6)}}, {1/{(3.5, 2.5)}}%
            }
            {1/{(1.5,2.5)}, 1/{(3.5,1.5)}}}; 
            \node (star) at (17.5,-4) {$*$};
            \node (star) at (17.5,-5) {$*$};
            \node (label) at (15.2, -6) {$L$};
        \end{tikzpicture}
    }%
    \caption{Another rule for the class $\Av(123)$.}
    \label{fig: another rule for Av(123)}
\end{figure}

As we keep track of occurrences of basis patterns through obstructions, the point cells do not add any information so we can remove them from our tilings. This is called \emph{factoring} in Albert \textit{et al.}~\cite{combexp}. 
We can apply the same factoring conditions that they use to our tilings, except that we can additionally factor cells in point rows.
Every point cell and the column it is in can be factored out from the tiling. By the third simplification, any row which contains point obstructions in every cell can be removed from a tiling. 
Figure~\ref{fig: factorised tilings} shows the factorisation of some tilings found so far for the class $\Av(123)$.

\begin{figure}[h]
    \centering
    \begin{tikzpicture}
        \node at (-2.2, 0) {};
        \node (start) at (-0.2,0) {
        \tiling{1.0}{2}{2}{0/0}%
        {%
        {1/{(0.5, 1.5)}}, {1/{(1.5, 1.5)}}, {2/{(1.45, 1.8), (1.25, 1.6)}}%
        }
        {1/{(1.5,1.5)}, 1/{(1.5,0.5)}}};
        \node (star) at (1,-0.5) {$*$};
        \node (label) at (0, -1.5) {$B$};
        
        \node (arrow) at (1.5,0) {$\cong$};
        
        \node (f insertion) at (2.5,0) {
            \tiling{1.0}{1}{1}{0/0}%
            {%
            }
            {}};
        \node (star) at (3.2,0) {$*$};
        \node (label A) at (2.5, -1) {$A$};
        
        \node (plus) at (3.7,0) {$\times$};
        
        \node (l insertion) at (4.7,0) {
        \tiling{1.0}{1}{2}{}%
        {%
        {1/{(0.5, 1.5)}}, {2/{(0.45, 1.8), (0.25, 1.6)}}%
        }
        {1/{(0.5,1.5)}, 1/{(0.5,0.5)}}};
        \node (star) at (5.4,-0.5) {$*$};
        \node (label) at (4.7, -1.5) {$M$};
        
        \node (star) at (7.6,-0.5) {};
    \end{tikzpicture}
    
    \begin{tikzpicture}
        \node at (-2.2, 0) {};
        \node (start) at (-0.2,0) {
            \tiling{1.0}{2}{1}{0/0, 1/0}%
            {%
            }
            {}};
        \node (star) at (1,0) {$*$};
        \node (label) at (0, -1) {$E$};
        
        \node (arrow) at (1.5,0) {$\cong$};
        
        \node (f insertion) at (2.5,0) {
            \tiling{1.0}{1}{1}{0/0}%
            {%
            }
            {}};
        \node (star) at (3.2,0) {$*$};
        \node (label A) at (2.5, -1) {$A$};
        
        \node (plus) at (3.7,0) {$\times$};
        
        \node (l insertion) at (4.7,0){
            \tiling{1.0}{1}{1}{0/0}%
            {%
            }
            {}};
        \node (star) at (5.4,0) {$*$};
        \node (label) at (4.7, -1) {$A$};
        
        \node (star) at (7.6,-0.5) {};
    \end{tikzpicture}
    
    \begin{tikzpicture}
        \node (start) at (-0.7,0)  {
        \tiling{1.0}{3}{2}{0/0, 1/0}%
        {%
        {1/{(0.5, 1.5)}}, {1/{(1.5, 1.5)}}, {2/{(2.45, 1.8), (2.25, 1.6)}}%
        }
        {1/{(2.5,1.5)}, 1/{(2.5,0.5)}}};
        \node (star) at (1,-0.5) {$*$};
        \node (label) at (-0.7, -1.5) {$F$};
        
        \node (arrow) at (1.5,0) {$\cong$};
        
        \node (f insertion) at (2.5,0) {
            \tiling{1.0}{1}{1}{0/0}%
            {%
            }
            {}};
        \node (star) at (3.2,0) {$*$};
        \node (label A) at (2.5, -1) {$A$};
        
        \node (plus) at (3.7,0) {$\times$};
        
        \node (l insertion) at (4.7,0){
            \tiling{1.0}{1}{1}{0/0}%
            {%
            }
            {}};
        \node (star) at (5.4,0) {$*$};
        \node (label) at (4.7, -1) {$A$};
        
        \node (plus) at (5.9,0) {$\times$};
        
        \node (l insertion) at (6.9,0){
        \tiling{1.0}{1}{2}{}%
        {%
        {1/{(0.5, 1.5)}}, {2/{(0.45, 1.8), (0.25, 1.6)}}%
        }
        {1/{(0.5,1.5)}, 1/{(0.5,0.5)}}};
        \node (star) at (7.6,-0.5) {$*$};
        \node (label) at (6.9, -1.5) {$M$};
    \end{tikzpicture}

    \begin{tikzpicture}
        \node at (-2.2, 0) {};
        \node (start) at (-0.2,0) {
        \tiling{1.0}{2}{2}{1/0}%
        {%
        {1/{(0.5, 0.5)}}, {1/{(1.5, 1.5)}}, {3/{(0.55, 1.8), (0.35, 1.6), (0.15, 1.4)}}%
        }
        {1/{(0.5,1.5)}}};
        \node (star) at (1,-0.5) {$*$};
        \node (label) at (0, -1.5) {$D$};
        
        \node (arrow) at (1.5,0) {$\cong$};
        
        \node (f insertion) at (2.5,0){
        \tiling{1.0}{1}{1}{}%
        {%
        {3/{(0.55, 0.8), (0.35, 0.6), (0.15, 0.4)}}}
        {%
        {1/{(0.5, 0.5)}}%
        }};
        \node (label A) at (2.5, -1) {$S$};
        
        \node (plus) at (3.5,0) {$\times$};
        
        \node (l insertion) at (4.5,0) {
            \tiling{1.0}{1}{1}{0/0}%
            {%
            }
            {}};
        \node (star) at (5.2,0) {$*$};
        \node (label) at (4.5, -1) {$A$};
        
        \node (star) at (7.6,-0.5) {};
    \end{tikzpicture}
    \caption{Examples of factorising tilings.}
    \label{fig: factorised tilings}
\end{figure}

Factoring always creates a collection of $1 \times 1$ tilings containing only point cells and at most one tiling with no point cells and a requirement list in every column, representing slots in a configuration which we call the \emph{factored tiling}. 
Additionally, as seen in Figure~\ref{fig: factorised tilings}, the factored tilings are often the same even though the tiling represents different configurations.
This allows us to reduce the number of tilings we need to consider. As tilings are also the states in the DFA created, this also reduces the number of states in the DFA we construct. 
In practice, the algorithm works directly on the factored tilings rather than considering the configurations they represent so the process is more efficient, but we will not go into the details here. For more information on how this works, please see Albert \textit{et al.}~\cite{combexp}.

By design factored tilings will never have more than two rows and for $k$-slot-bounded classes will never have more than $k$ columns. 
Moreover, as the basis patterns are a fixed size and there is a finite number of ways in which they can appear on a tiling, there are finitely many tilings of a given size.
This proves that continuing this process will eventually terminate and result in a DFA which recognises the vertical insertion encodings of the slot-bounded Cayley permutation class of interest.
In the next section we give a full example of a DFA constructed using this method for the class $\Av(231, 312, 2121)$.

\subsection{Implementation of the algorithm}
\label{sec:algo-code}

We have implemented the algorithm described in Section~\ref{sec:tilings} in Python. 
For the horizontal insertion encoding, which will be described in Section~\ref{sec:left_to_right}, we have also implemented a similar algorithm. Both implementations can be found on GitHub \cite{Bean_cperms_ins_enc_2025} and can be used to enumerate any class of Cayley permutations with a set of vertical or horizontal insertion encodings that form a regular language. For example, as discussed in Section~\ref{sec:intro}, Cerbai~\cite{Cerbai2021} found the class of hare pop-stack sortable Cayley permutations to be $\Av(231, 312, 2121)$. Our implementation of the algorithm can be used to find a set of 9 rules for this class shown in Figures~\ref{fig: rule 1 of spec, S}, \ref{fig: rule 2 of spec, A}, \ref{fig: rule 3 of spec, B}, \ref{fig: rule 4 of spec, C}, \ref{fig: rule 5 of spec, D}, \ref{fig: rule 6 of spec, E}, \ref{fig: rule 7 of spec, F}, \ref{fig: rule 8 of spec, G}, and \ref{fig: rule 9 of spec, H}. Each figure shows the rule with unfactored tilings to make it clear which insertion corresponds to which tiling and with factored tilings to make it clear which fully factored tilings are of interest for the class. The corresponding rule in the grammar is also shown.

\begin{figure}


    \caption{$S \to f_{1,1} \: |\: \ell_{1,1}A \:|\: r_{1,1} B \:|\: m_{1,1} C$}
    \label{fig: rule 1 of spec, S}
\end{figure}

\begin{figure}        
    %
    \caption{$A \to  f_{1,1} \: |\: \ell_{1,1}A \:|\: r_{1,1} B \:|\: m_{1,1} C \:|\: f_{1,0} \:|\: \ell_{1,0}A \:|\: r_{1,0} B \:|\: m_{1,0} C$}
    \label{fig: rule 2 of spec, A}
\end{figure}

\begin{figure}

    %
    \caption{$B \to f_{1,1} \:|\: \ell_{1,1}D \:|\: r_{1,1} B \:|\: m_{1,1} E$}
    \label{fig: rule 3 of spec, B}
\end{figure}

\begin{figure}
    %
    
    \caption{$ C \to f_{2,0} B \:|\: \ell_{2,0}C \:|\:  f_{1,1} A \:|\: \ell_{1,1}F \:|\: r_{1,1}G \:|\: m_{1,1}H $}
    \label{fig: rule 4 of spec, C}
\end{figure}

\begin{figure}
    \centering
    %
    \caption{$D \to f_{1,0} \:|\: \ell_{1,0}D $}
    \label{fig: rule 5 of spec, D}
\end{figure}

\begin{figure}
    \centering
    %
    \caption{$E \to f_{2,0} B \:|\: \ell_{2,0}E$}
    \label{fig: rule 6 of spec, E}
\end{figure}

\begin{figure}
    \centering
    %
    \caption{$H \to f_{1,0} A \:|\: \ell_{1,0}H $}
    \label{fig: rule 7 of spec, F}
\end{figure}

\begin{figure}

    %
    
    \caption{$G \to f_{1,1} A \:|\: \ell_{1,1}F \:|\: r_{1,1}G \:|\: m_{1,1}H $}
    \label{fig: rule 8 of spec, G}
\end{figure}

\begin{figure}
    \centering
    %

    \caption{$H \to f_{2,0} F \:|\: \ell_{1,0}F$}
    \label{fig: rule 9 of spec, H}
\end{figure}

These 9 rules can be turned into the following grammar for the class.
\begin{equation*}
    \begin{split}
        S & \to f_{1,1} \: |\: \ell_{1,1}A \:|\: r_{1,1} B \:|\: m_{1,1} C                                                                  \\
        A & \to  f_{1,1} \: |\: \ell_{1,1}A \:|\: r_{1,1} B \:|\: m_{1,1} C \:|\: f_{1,0} \:|\: \ell_{1,0}A \:|\: r_{1,0} B \:|\: m_{1,0} C \\
        B & \to f_{1,1} \:|\: \ell_{1,1}D \:|\: r_{1,1} B \:|\: m_{1,1} E                                                                   \\
        C & \to f_{2,0} B \:|\: \ell_{2,0}C \:|\:  f_{1,1} A \:|\: \ell_{1,1}F \:|\: r_{1,1}G \:|\: m_{1,1}H                                \\
        D & \to f_{1,0} \:|\: \ell_{1,0}D                                                                                                   \\
        E & \to f_{2,0} B \:|\: \ell_{2,0}E                                                                                                 \\
        F & \to f_{1,0} A \:|\: \ell_{1,0}F                                                                                                 \\
        G & \to f_{1,1} A \:|\: \ell_{1,1}F \:|\: r_{1,1}G \:|\: m_{1,1}H                                                                   \\
        H & \to f_{2,0} F \:|\: \ell_{1,0}F                                                                                                 \\
    \end{split}
\end{equation*}

Let $S(x)$ be the generating function for $S$, $A(x)$ be the generating function for $A$, and so on. Then we have the following system of equations.

\begin{equation*}
    \begin{split}
        S(x) & = x \: +\: xA(x) \:+\: x B(x) \:+\: x C(x)                                              \\
        A(x) & = x \: +\: xA(x) \:+\: x B(x) \:+\: x C(x) \:+\: x \:+\: xA(x) \:+\: x B(x) \:+\: xC(x) \\
        B(x) & = x \:+\: xD(x) \:+\: x B(x) \:+\: x E(x)                                               \\
        C(x) & = x B(x) \:+\: xC(x) \:+\: x A(x) \:+\: xF(x) \:+\:xG(x) \:+\:xH(x)                     \\
        D(x) & = x \:+\: xD(x)                                                                         \\
        E(x) & = x B(x) \:+\: xE(x)                                                                    \\
        F(x) & = x A(x) \:+\: xF(x)                                                                    \\
        G(x) & = x A(x) \:+\: xF(x) \:+\:xG(x) \:+\:xH(x)                                              \\
        H(x) & = x A(x) \:+\: xF(x)                                                                    \\
    \end{split}
\end{equation*}

This solves to give the generating function $S(x)$ for $\Av(231, 312, 2121)$
\begin{equation*}
    \begin{split}
        S(x) & = \frac{2x^{2} - 2x^{3} - x}{4x^{3} - 6x^{2} + 5x - 1}.
    \end{split}
\end{equation*}
This is the generating function for the sequence beginning $1, 3, 11, 41, 151, 553, 2023, 7401, 27079, 99081$ which is sequence A335793 in the OEIS~\cite{oeis}.

\section{Horizontal insertion encoding}
\label{sec:left_to_right}

Unlike for permutations, Cayley permutations are not symmetric under inverses so the rules that govern inserting new rightmost values are different from those of inserting new maximum values. In this section, we will define the \emph{horizontal insertion encoding} that instead traces how Cayley permutations are built from left to right. This is strictly different from the vertical insertion encoding for Cayley permutations although the vertical and horizontal versions are equivalent under the inverse symmetry for permutations.

In this section, we will also classify when the set of horizontal insertion encodings of a class is regular. A similar algorithm as in Section~\ref{sec:regular-algo} can be used to enumerate Cayley permutation classes with horizontal insertion encodings which form a regular language. Although we will not describe it in this paper, we have an implementation for this on GitHub~\cite{Bean_cperms_ins_enc_2025} as well.

The set of Cayley permutations can be built by iteratively adding new rightmost values. As Cayley permutations can have repeated values, the insertion of the new rightmost number can be either a new value or a repeated value already in the Cayley permutation. The \emph{(horizontal) evolution} traces how a Cayley permutation is built in this way, starting from the empty Cayley permutation. Again, we use $\diamond$ to denote where future points will be inserted. To denote that the value to come is a repeated value, we use $\overline{\diamond}$. We call $\diamond$ a \emph{new slot} and $\overline{\diamond}$ a \emph{repeating slot}. For example, the evolution of the Cayley permutation 243441 is shown in Figure~\ref{fig: 243441 hori evolution}.

\begin{figure}[h]
    \begin{center}
        \begin{tikzpicture}
            \node at (0,0) {$\diamond$};
            
            \node at (2em,0) {$\rightarrow$};
            
            \node at (4em, 0) {$1$};
            \node at (5em, 1em) {$\diamond$};
            \node at (5em, -1em) {$\diamond$};
            
            \node at (7em, 0) {$\rightarrow$};
            
            \node at (9em, 0) {$1$};
            \node at (9.6em, 2em) {$2$};
            \node at (10.6em, -1em) {$\diamond$};
            \node at (10.6em, 1em) {$\diamond$};
            \node at (10.6em, 2em) {$\overline{\diamond}$};
            
            \node at (12.6em, 0) {$\rightarrow$};
            
            \node at (14.6em, 0) {$1$};
            \node at (15.2em, 2em) {$3$};
            \node at (15.8em, 1em) {$2$};
            \node at (16.8em, -1em) {$\diamond$};
            \node at (16.8em, 2em) {$\overline{\diamond}$};
            
            \node at (18.8em, 0) {$\rightarrow$};
            
            \node at (20.8em, 0) {$1$};
            \node at (21.4em, 2em) {$3$};
            \node at (22em, 1em) {$2$};
            \node at (22.6em, 2em) {$3$};
            \node at (23.6em, -1em) {$\diamond$};
            \node at (23.6em, 2em) {$\overline{\diamond}$};
            
            \node at (25.6em, 0) {$\rightarrow$};
            
            \node at (27.6em, 0) {$1$};
            \node at (28.2em, 2em) {$3$};
            \node at (28.8em, 1em) {$2$};
            \node at (29.4em, 2em) {$3$};
            \node at (30em, 2em) {$3$};
            \node at (31em, -1em) {$\diamond$};
            
            \node at (33em, 0) {$\rightarrow$};
            
            \node at (35em, 0) {$2$};
            \node at (35.6em, 2em) {$4$};
            \node at (36.2em, 1em) {$3$};
            \node at (36.8em, 2em) {$4$};
            \node at (37.4em, 2em) {$4$};
            \node at (38em, -1em) {$1$};
            
        \end{tikzpicture}
    \end{center}
    \caption{The horizontal evolution of the Cayley permutation 243441.}
    \label{fig: 243441 hori evolution}
\end{figure}

Again, we call each step of the evolution a \emph{configuration} and call the points in a configuration $c$ which have already been placed the \emph{prefix} of $c$. The prefix of a configuration always forms a Cayley permutation. An insertion in a conconfiguration is determined by the index of the slot inserted into, the way the slot is effected by the insertion, and whether or not there are more of the same value to be inserted. 

We use the letters $u$, $d$, $m$, and $f$ denoting up, down, middle, and fill respectively. 
For each insertion, we use the letter $a_{i, j}$ with $a \in \{u, d, m, f\}$, $i \in \mathbb{N}_1$ as the index of the slot being inserted into counting from bottom to top, and $j \in \{0, 1\}$ to tell us if there are more of the same value still to be inserted. Only $f$ can be inserted into a repeating slot. Table~\ref{tab:horizontal-insertions} shows all possible insertions of a value $n$ into a slot at index $i$.
\begin{table}[h]
    \centering
    \begin{tabular}{|ccccc|ccccc|}
        \hline
        $u_{i, 0}$: & $\diamond$            & $\rightarrow$ & $n$ &            & $u_{i, 1}$: & $\diamond$            & $\rightarrow$ & $n$ & $\overline{\diamond}$ \\
                    &                       &               &     & $\diamond$ &             &                       &               &     & $\diamond$            \\
        \hline
                    &                       &               &     & $\diamond$ &             &                       &               &     & $\diamond$            \\
        $m_{i, 0}$: & $\diamond$            & $\rightarrow$ & $n$ &            & $m_{i, 1}:$ & $\diamond$            & $\rightarrow$ & $n$ & $\overline{\diamond}$ \\
                    &                       &               &     & $\diamond$ &             &                       &               &     & $\diamond$            \\
        \hline
                    &                       &               &     & $\diamond$ &             &                       &               &     & $\diamond$            \\
        $d_{i, 0}$: & $\diamond$            & $\rightarrow$ & $n$ &            & $d_{i, 1}$: & $\diamond$            & $\rightarrow$ & $n$ & $\overline{\diamond}$ \\
        \hline
        $f_{i, 0}$: & $\diamond$            & $\rightarrow$ & $n$ &            & $f_{i, 1}$: & $\diamond$            & $\rightarrow$ & $n$ & $\overline{\diamond}$ \\
        \hline
        $f_{i, 0}$: & $\overline{\diamond}$ & $\rightarrow$ & $n$ &            & $f_{i, 1}$: & $\overline{\diamond}$ & $\rightarrow$ & $n$ & $\overline{\diamond}$ \\
        \hline
    \end{tabular}
    \caption{Types of insertions of the value $n$ into a slot at index $i$. Only $f_{i, 0}$ and $f_{i, 1}$ can insert into $\overline{\diamond}$.}
    \label{tab:horizontal-insertions}
\end{table}

For example, the horizontal insertion encoding for the Cayley permutation $243441$ is the word $m_{1,0}u_{2,1}f_{2,0}f_{2,1}f_{2,0}f_{1,0}$ and the evolution is shown in Figure~\ref{fig: 243441 hori evolution, with encoding}.

\begin{figure}[h]
    \begin{center}
        \begin{tikzpicture}
            \node at (0,0) {$\diamond$};
            
            \node at (2em,0) {$\rightarrow$};
            \node at (2em,1em) {$m_{1, 0}$};
            
            \node at (4em, 0) {$1$};
            \node at (5em, 1em) {$\diamond$};
            \node at (5em, -1em) {$\diamond$};
            
            \node at (7em, 0) {$\rightarrow$};
            \node at (7em,1em) {$u_{2, 1}$};
            
            \node at (9em, 0) {$1$};
            \node at (9.6em, 2em) {$2$};
            \node at (10.6em, -1em) {$\diamond$};
            \node at (10.6em, 1em) {$\diamond$};
            \node at (10.6em, 2em) {$\overline{\diamond}$};
            
            \node at (12.6em, 0) {$\rightarrow$};
            \node at (12.6em,1em) {$f_{2, 0}$};
            
            \node at (14.6em, 0) {$1$};
            \node at (15.2em, 2em) {$3$};
            \node at (15.8em, 1em) {$2$};
            \node at (16.8em, -1em) {$\diamond$};
            \node at (16.8em, 2em) {$\overline{\diamond}$};
            
            \node at (18.8em, 0) {$\rightarrow$};
            \node at (18.8em,1em) {$f_{2, 1}$};
            
            \node at (20.8em, 0) {$1$};
            \node at (21.4em, 2em) {$3$};
            \node at (22em, 1em) {$2$};
            \node at (22.6em, 2em) {$3$};
            \node at (23.6em, -1em) {$\diamond$};
            \node at (23.6em, 2em) {$\overline{\diamond}$};
            
            \node at (25.6em, 0) {$\rightarrow$};
            \node at (25.6em,1em) {$f_{2, 0}$};
            
            \node at (27.6em, 0) {$1$};
            \node at (28.2em, 2em) {$3$};
            \node at (28.8em, 1em) {$2$};
            \node at (29.4em, 2em) {$3$};
            \node at (30em, 2em) {$3$};
            \node at (31em, -1em) {$\diamond$};
            
            \node at (33em, 0) {$\rightarrow$};
            \node at (33em,1em) {$f_{1, 0}$};
            
            \node at (35em, 0) {$2$};
            \node at (35.6em, 2em) {$4$};
            \node at (36.2em, 1em) {$3$};
            \node at (36.8em, 2em) {$4$};
            \node at (37.4em, 2em) {$4$};
            \node at (38em, -1em) {$1$};
            
        \end{tikzpicture}
    \end{center}
    \caption{The horizontal evolution of the Cayley permutation 243441 with the encoding shown with each insertion.}
    \label{fig: 243441 hori evolution, with encoding}
\end{figure}

As for vertical evolutions, for a configuration $c$, we call the Cayley permutations which have $c$ as a configuration in its horizontal evolution the derivations of $c$, denoted $D(c)$.

Once again, for any configuration with $s$ slots, the minimum number of insertions needed to create a Cayley permutation is $s$. Therefore, for any Cayley permutation class with a set of horizontal insertion encodings which is regular there exists a bound $k$ on the number of slots in any configuration leading to a Cayley permutation in the class. We call such classes $\emph{(horizontal) slot-bounded}$. Being slot-bounded is also a sufficient condition for a Cayley permutation class to have a set of horizontal insertion encodings which is regular. This can be shown by adapting the algorithm in Section~\ref{sec:regular-algo} to the horizontal case, which we will not do here for brevity. In this section, we will give a characterisation of the slot-bounded Cayley permutation classes. 

Let $\SBH(k)$ be the set of Cayley permutations with evolutions that contain configurations with at most $k$ slots. We will show that, as with the vertical case, for each $k$ this is a finitely-based Cayley permutation class by proving certain properties of configurations with $k$ slots, leading to Proposition~\ref{prop:SBH-basis} which describes the basis of $\SBH(k)$.

First, note that if a Cayley permutation has more than $k$ slots in its evolution then it must contain a configuration with exactly $k + 1$ slots since it is only possible to reduce the number of slots by one with an $f_{i, 0}$ insertion.
For a configuration with $k+1$ slots, the slots could either be a $\diamond$ or a $\overline{\diamond}$ slot. For every repeating slot, there is at least one occurrence of the same value in the prefix of the configuration and for every pair of new slots there must be at least one occurrence of a value between them in the prefix, as shown in Figure~\ref{fig:slots-imply-points}.

\begin{figure}[h]
    \begin{center}
        \begin{tikzpicture}
            \foreach \x/\y in {0/3, 0/2.5, 2.5/3, 2.5/2.5}
            \node [font=\normalsize] at (\x,\y) {$\diamond$};
            \foreach \x/\y in {0/3.7, 0/2, 2.5/3.7, 2.5/2, 4.9/3.5, 4.9/2.2, 7.4/3.5, 7.4/2.2}
            \node [font=\normalsize] at (\x,\y) {$\vdots$};
            \foreach \x/\y in {4.9/2.75, 7.4/2.75}
            \node [font=\normalsize] at (\x,\y) {$\overline{\diamond}$};
            \foreach \x/\y in {2/2.75, 7/2.75}
            \node [font=\normalsize] at (\x,\y) {$n$};
            \foreach \x/\y in {1/2.75, 6/2.75}
            \node [font=\normalsize] at (\x,\y) {$\implies$};
            \node [font=\normalsize] at (3.75, 2.75) {and};
        \end{tikzpicture}
        \caption{Slots imply the existence of values in the prefix of a configuration.}
        \label{fig:slots-imply-points}
    \end{center}
\end{figure}

The set $\SBH(k)$ can be defined by avoiding all derivations of configurations with $k + 1$ slots. To give a characterisation of these configurations, we will first define \emph{subconfigurations}. Let $c$ be a configuration and $c'$ be a configuration that can be obtained from $c$ by deleting elements of the prefix of $c$ (and standardising if necessary). In this process of deleting points, repeating slots become new slots if all of the corresponding values in the prefix are removed.
If enough points are removed from the prefix then slots may also be standardised to have different heights, but the number of slots and the relative heights of the slots will not change.
We say that $c'$ is a \emph{subconfiguration} of $c$ and we have the following useful result. Note, in our definition, we will have the same number of slots at the same relative values as the configuration, as we are only allowing removal of values in the prefix. Removing some values in the prefix could result in something which is not a configuration, in which case, we do not consider this a subconfiguration. 

\begin{lemma}\label{lem:subconfig}
    Let $c$ be a configuration and $c'$ be a subconfiguration of $c$. Then for every derivation $\sigma \in D(c)$ there exists a derivation $\sigma' \in D(c')$ such that $\sigma$ contains $\sigma'$. 
\end{lemma}

\begin{proof}
    Let $c$ be a configuration with insertion encoding $w$ and $c'$ be the subconfiguration of $c$ with insertion encoding $w'$. Then, the insertion encoding of a Cayley permutation $\sigma \in D(c)$ can be written as $wv$. Let $\sigma' \in D(c')$ be the Cayley permutation with insertion encoding $w'v$. As the prefix of $c'$ is contained in the prefix of $c$ at the same values, and the height of the slots are the same, it follows that $\sigma$ contains $\sigma'$. 
\end{proof}

Given a configuration $c$ and subconfiguration $c'$ whose derivations must be avoided, this result tells us that avoiding the derivations $D(c)$ is implied by avoiding the derivations of $D(c')$. We're now ready to state and prove the result. 

\begin{proposition}
    \label{prop:SBH-basis}
    For each integer $k$, let $C_k$ be the set of configurations with $k + 1$ slots such that the following conditions are fulfilled.
    \begin{itemize}
        \item The prefix of $c$ is a permutation.
        \item There are no values in the prefix of $c$ larger than the top slot or smaller than the bottom slot.
        \item There are no values in the prefix which are strictly between a repeating slot and another slot.
        \item There are no three consecutive repeating slots in $c$.
        \item There are no two consecutive repeating slots at the top or bottom.
    \end{itemize}
    Then $\SBH(k)$ is the Cayley permutation class whose basis is the set of minimal Cayley permutations in the derivations of each configuration in $c \in C_k$ of size $|c|$. 
\end{proposition}
\begin{proof}
    The basis elements of $\SBH(k)$ are the minimal Cayley permutations which have an evolution containing a configuration with $k+1$ slots, which is a superset of $C_k$. We will prove that for any configuration which has $k+1$ slots that does not satisfy the conditions there exists a smaller subconfiguration. By Lemma~\ref{lem:subconfig}, avoiding the derivations of a configuration is implied by avoiding the derivations of a subconfiguration, so we can ignore the configuration and work with the subconfiguration instead.  
    
    Let $c$ be a configuration with $k + 1$ slots. 
    
    \begin{itemize}
        \item Assume that the prefix of $c$ is not a permutation. Let $c'$ be the subconfiguration obtained by deleting every index that is not the first occurrence of each value. Then $c'$ is a subconfiguration of $c$ with a prefix that is a permutation.
        \item Assume there are values in the prefix of $c$ larger than the topslot or smaller than the bottom slot. Then let $c'$ be the subconfiguration obtained by deleting these values.
        \item Assume that there is value $n$ in the prefix strictly between a repeating slot and another slot, as shown in Figure~\ref{fig:extra-conditions}. Let $c'$ be the subconfiguration obtained by deleting all occurrences of the value $n$ in the prefix.
        \item Assume there are three consecutive repeating slots with values $n$, $n+1$, and $n + 2$ respectively, as shown in Figure~\ref{fig:3repeating-slots}. Let $c'$ be the configuration obtained by deleting all occurrences of $n + 1$. Note, the repeating slot at this height will become a new slot. Clearly, $c'$ is a smaller subconfiguration of $c$.
              
        \item Assume there are two consecutive repeating slots at the top of the configuration with values $n$ and $n + 1$, as shown in Figure~\ref{fig:2repeating-slots}. Then let $c'$ be the subconfiguration obtained by deleting all occurrence of $n + 1$, and replacing the repeating slot with a new slot. Similarly, if there are two consecutive repeating slots at the bottom with values $1$ and $2$, then we take $c'$ to be the subconfiguration obtained by deleting all occurrences of $1$ and replacing this repeating slot with a new slot.
    \end{itemize}
    
    \begin{figure}[h]
        \centering
        \begin{subfigure}{0.55\textwidth}
            \centering
            \begin{tikzpicture} 
                \foreach \x/\y in {1/2.9,3/2.7}
                \node  at (\x,\y) {$\diamond$};
                \foreach \x/\y in {1/1.6,1/3.5,3/1.9,3/3.3,6/1.6,6/3.5,8/1.9,8/3.3}
                \node  at (\x,\y) {$\vdots$};
                \foreach \x/\y in {1/2.1,3/2.3,6/2.1,6/2.9,8/2.3,8/2.7}
                \node  at (\x,\y) {$\overline{\diamond}$};
                \foreach \x/\y in {0.3/2.5,5.3/2.5}
                \node  at (\x, \y) {$n$};
                \foreach \x/\y in {4/2.5}
                \node  at (\x, \y) {and};
                \foreach \x/\y in {1.8/2.5, 6.8/2.5}
                \node  at (\x,\y) {$\leadsto$};
            \end{tikzpicture}
            \subcaption{Values between a repeating slot and another slot.}
            \label{fig:extra-conditions}
        \end{subfigure}
        \begin{subfigure}{0.45\textwidth}
            \centering
            \begin{tikzpicture} 
                \foreach \x/\y in {4/2.5}
                \node  at (\x,\y) {$\diamond$};
                \foreach \x/\y in {1/1.6, 1/3.6, 4/1.6, 4/3.6}
                \node  at (\x,\y) {$\vdots$};
                \foreach \x/\y in {1./2, 1/2.5, 1./3, 4/2, 4/3}
                \node  at (\x,\y) {$\overline{\diamond}$};
                \node  at (0, 3) {$n+2$};
                \node  at (3.3, 3) {$n+2$};
                \foreach \x/\y in {0/2.5}
                \node  at (\x, \y) {$n+1$};
                \foreach \x/\y in {0/2, 3.3/2}
                \node  at (\x, \y) {$n$};
                
                \foreach \x/\y in {1.8/2.5}
                \node  at (\x,\y) {$\leadsto$};
            \end{tikzpicture}
            \caption{Three consecutive repeating slots.}
            \label{fig:3repeating-slots}
        \end{subfigure}
        \begin{subfigure}{0.45\textwidth}
            \centering
            \begin{tikzpicture} 
                \foreach \x/\y in {12/3}
                \node  at (\x,\y) {$\diamond$};
                \foreach \x/\y in {10/2.1, 12/2.1}
                \node  at (\x,\y) {$\vdots$};
                \foreach \x/\y in {10/2.5, 10/3, 12/2.5}
                \node  at (\x,\y) {$\overline{\diamond}$};
                \foreach \x/\y in { 9.3/3}
                \node  at (\x, \y) {$n+1$};
                \foreach \x/\y in {9.3/2.5, 11.5/2.5}
                \node  at (\x, \y) {$n$};
                
                \foreach \x/\y in {10.8/2.5}
                \node  at (\x,\y) {$\leadsto$};
            \end{tikzpicture}
            \caption{Two consecutive repeating slots at the top of a configuration.}
            \label{fig:2repeating-slots}
        \end{subfigure}
        \caption{Configurations which do not satisfy the conditions of Proposition~\ref{prop:SBH-basis} have smaller subconfigurations.}
    \end{figure}
    
    In each case, we reduce the size of the configuration, so the process of repeatedly replacing $c$ with $c'$ will terminate with a configuration in $C_k$. Hence, $\SBH(k)$ is defined by avoiding all of the derivations of the configurations in $C_k$. 
    %
\end{proof}

To compute the basis of $\SBH(k)$ we consider all configurations with $k + 1$ slots in which the prefix forms a permutation which we call $\pi_L$. Any derivation of such a configuration will contain a derivation $\sigma$ of size $|\pi_L| + k + 1$ from the same configuration. The values to the right of the prefix $\pi_L$ will form a permutation of size $k+1$ which we denote $\pi_R$. As there can not be three consecutive repeating slots, there are no three consecutive values in $\pi_L$ or $\pi_R$. 
This gives a bound on the size of the Cayley permutations in the basis of $\SBH(k)$, stated in Remark~\ref{rmk:bounds-on-pi_L}.

\begin{remark}
    \label{rmk:bounds-on-pi_L}
    A basis element $\pi_L\pi_R$ in $\SBH(k)$ will have $|\pi_R|= k + 1$, $\lfloor \frac{k+1}{2} \rfloor \leq |\pi_L|$ due to the restriction on repeating slots and $|\pi_L| \leq k$ which is achieved if all slots are new slots. 
\end{remark}

A slot-bounded Cayley permutation class is a subclass of $\SBH(k)$ for some $k$. In order to characterise the slot-bounded Cayley permutation classes we need some definitions. 

A size $n + k$  permutation $\pi$ is a \emph{horizontal juxtaposition} of the size $n$  permutation $\sigma$ and the size $k$  permutation $\tau$ if the size $n$ prefix of $\pi$ standardises to $\sigma$ and the size $k$ suffix of $\pi$ standardises to $\tau$. For example, a horizontal juxtaposition of $1423$ and $312$ is $1624735$. 

There are four classes of horizontal juxtapositions which we will need for our theory. They come from horizontally juxtaposing strictly increasing and strictly decreasing sequences. We denote these classes $\mathcal{H}_{I,I}$, $\mathcal{H}_{I,D}$, $\mathcal{H}_{D,I}$, and $\mathcal{H}_{D,D}$. For example, $13245$ is a juxtaposition of the strictly increasing sequence $12$ and the strictly increasing sequence $123$, so is in $\mathcal{H}_{I,I}$, 
and $521346$ is a juxtaposition of the strictly decreasing sequence $21$ and the strictly increasing sequence $1234$, so is in $\mathcal{H}_{D,I}$, as shown in Figure~\ref{fig:examples-of-horizontal-jux}.

\begin{figure}[h]
    \begin{center}
        \resizebox{10cm}{!}{
            \begin{tikzpicture}
                \draw[step=1cm,gray,thin] (0.5,0.5) grid (5.5,5.5);
                \foreach \x/\y in {1/1, 2/3, 3/2, 4/4, 5/5}
                \fill[black] (\x,\y) circle (0.14cm);
                \draw[thick, dashed] (2.5, 0.5) -- (2.5, 5.5);
                \node[font = \Large] at (3, 0) {$13245 \in \mathcal{H}_{I,I}$ is a horizontal};
                \node[font = \Large] at (3, -0.7) {juxtaposition of $12$ and $123$};
            \end{tikzpicture}
            \phantom{----}
            \begin{tikzpicture}
                \draw[step=1cm,gray,thin] (0.5,0.5) grid (6.5,6.5);
                \foreach \x/\y in {1/5, 2/2, 3/1, 4/3, 5/4, 6/6}
                \fill[black] (\x,\y) circle (0.14cm);
                \draw[thick, dashed] (2.5, 0.5) -- (2.5, 6.5);
                \node[font = \Large] at (3, 0) {$521346 \in \mathcal{H}_{D,I}$ is a horizontal};
                \node[font = \Large] at (3, -0.7) {juxtaposition of $21$ and $1234$};
            \end{tikzpicture}
        }
    \end{center}
    \caption{The horizontal juxtapositions $13124 \in \mathcal{H}_{I,I}$ and $421234 \in \mathcal{H}_{D,I}$.}
    \label{fig:examples-of-horizontal-jux}
\end{figure}

A \emph{horizontal alternation} is a permutation of size $2n$ for which the prefix of size $n$ is a permutation of the odd values and the suffix of size $n$ is a permutation of the even values. For example, $135264$ is a horizontal alternation. By symmetry Lemma~\ref{lem: alternation implies juxtaposition} can be immediately adapted to Lemma~\ref{lem: hori alternation implies juxtaposition} for horizontal alternations.

\begin{lemma}
    \label{lem: hori alternation implies juxtaposition}
    For each $\mathcal{H} \in \{\mathcal{H}_{I,I}, \mathcal{H}_{I,D}, \mathcal{H}_{D,I}, \mathcal{H}_{D,D}\}$, the size $2n$ horizontal alternation in $\mathcal{H}$ contains every horizontal juxtaposition in $\mathcal{H}$ of size up to $n$.
\end{lemma}

We are ready to prove the main theorem of this section, Theorem \ref{thm:slot-bounded-horizontal}.

\begin{theorem}
    \label{thm:slot-bounded-horizontal}
    A Cayley permutation class $\Av(B)$ is a subclass of $\SBH(k)$ for some $k$ if and only if $B$ contains a permutation from each of the classes $\mathcal{H}_{I,I}$, $\mathcal{H}_{I,D}$, $\mathcal{H}_{D,I}$, and $\mathcal{H}_{D,D}$.
\end{theorem}

\begin{proof}
    Suppose that $\Av(B) \subseteq \SBH(k)$ for some $k$. Then $\Av(B)$ avoids derivations from the configurations in $C_k$, as stated in Proposition~\ref{prop:SBH-basis}. In particular, consider a configuration $c$ in $C_k$ with no repeating slots. There is a value between every pair of slots, so the $|c| = 2k + 1$. One of these configurations will have the values in strictly increasing order and another will have the values in strictly decreasing order. A derivation of size $2k + 1$ of these configurations will be formed by inserting values in strictly increasing or strictly decreasing order, creating a permutation from each of the classes $\mathcal{H}_{I,I}$, $\mathcal{H}_{I,D}$, $\mathcal{H}_{D,I}$, and $\mathcal{H}_{D,D}$. 
    As they are not in $\Av(B)$, $B$ must contain either the permutation or a subpermutation from each of these four classes. But, as each of these classes are downwards closed, any subpermutation of a permutation in the class is also in the class. Therefore, $B$ must contain a permutation from each of $\mathcal{H}_{I,I}$, $\mathcal{H}_{I,D}$, $\mathcal{H}_{D,I}$, and $\mathcal{H}_{D,D}$.
    
    Suppose $\Av(B)$ is a Cayley permutation class with $B$ containing permutations from each of the four classes of horizontal juxtapositions.
    Suppose the length of the largest horizontal juxtaposition in $B$ is $n$. We will show that there exists a nonnegative integer $k$ such that every basis element of $\SBH(k)$ contains a horizontal juxtaposition that is in $B$, so $\Av(B) \subseteq \SBH(k)$ for this $k$. 
    
    Let $k = 3n^4$ and $\sigma$ be any basis element in $\SBH(k)$. Although a smaller $k$ may work, we will show that $k = 3n^4$ is sufficient. We can write $\sigma = \pi_L\pi_R$ with $|\pi_R| = k + 1$ and $\pi_L = |\sigma| - k - 1$. From our discussion about the basis of $\SBH(k)$ it follows that $\pi_L$ and $\pi_R$ are permutations, and that $\sigma$ is a derivation of some configuration $c \in C_k$. As $c$ contains no three consecutive repeating slots, there must be at least $\lfloor \frac{k + 1}{3} \rfloor$ new slots in the configuration, which correspond to values in $\pi_R$ that are not in $\pi_L$. As $k = 3n^4$, this is at least $n^4$ values. By the Erd\H{o}s-Szekeres theorem~\cite{Erdos1935}, there is a strictly increasing or strictly decreasing subsequence of size $n^2$ in $\pi_R$, which we call $\tau$. 
    
    By Remark~\ref{rmk:bounds-on-pi_L}, there is at least $\lfloor \frac{k+1}{2} \rfloor = \lfloor \frac{3n^4 + 1}{2} \rfloor$ values in $\pi_L$. All of the values in $\pi_L$ are not in $\tau$ and since $n^2 < \lfloor \frac{3n^4 + 1}{2} \rfloor$, we can take $n^2$ values from $\pi_L$ which perfectly interleave with $\tau$. Of these $n^2$ values, by the Erd\H{o}s-Szekeres theorem~\cite{Erdos1935}, there is a strictly increasing or strictly decreasing subsequence of size $n$, say $\alpha$. We can take a subset of the values in $\tau$ that perfectly interleaves this subsequence, say $\beta$. The concatenation of $\alpha$ and $\beta$ form a horizontal alternation of size $2n$ which is contained in $\sigma$ in one of $\mathcal{H}_{I,I}$, $\mathcal{H}_{I,D}$, $\mathcal{H}_{D,I}$, or $\mathcal{H}_{D,D}$. By Lemma~\ref{lem: hori alternation implies juxtaposition}, this horizontal alternation contains every horizontal juxtaposition from the same class of size up to $n$, specifically a horizontal juxtaposition in $B$. Therefore, the basis element $\sigma$ of $\SBH(k)$ contains a basis element in $B$. As $\sigma$ was arbitrary, every basis element of $\SBH(k)$ contains a basis element in $B$, so none of them are in $\Av(B)$, hence $\Av(B) \subseteq \SBH(k)$.
\end{proof}

An algorithm similar to the one given in Section~\ref{sec:regular-algo} can be used for computing a DFA that accepts the horizontal insertion encodings of any finitely-based slot-bounded Cayley permutation class. Hence, if a finitely-based Cayley permutation class avoids arbitrarily long horizontal juxtapositions then its horizontal insertion encodings form a regular language and a DFA can be computed for it.

There are 13 Cayley permutation classes which avoid two size 3 patterns and have a set of horizontal insertion encodings which form a regular language, as shown in Table~\ref{tab:results for size 3}. 
For bases $B$ and $B'$, obtained by taking the complement or reverse of all Cayley permutations in $B$, the size $n$ Cayley permutations in $\Av(B)$ are in bijection with all size $n$ Cayley permutations in $\Av(B')$, so $\Av(B)$ and $\Av(B')$ have the same generating function. 
These 13 Cayley permutation classes correspond to 5 distinct symmetry classes.
Their generating functions and first 10 terms have been computed by an implementation of our algorithm~\cite{Bean_cperms_ins_enc_2025} and are given in Table~\ref{tab:two-size-3-generating-functions} with the corresponding references to the OEIS~\cite{oeis}.

\begin{table}[h]
    \centering
    \begin{tabular}{|c|c|c|c|c|}
        \hline
        Bases         & Generating function                                                                                                                            & First 10 terms                                & OEIS    \\
        \hline
        $Av(123,321)$ & $\frac{4 x^{4} - 2 x^{3} + 7 x^{2} - 4 x + 1}{\left(x - 1\right)^{3} \left(2 x - 1\right)}$                                                    & 1, 1, 3, 11, 37, 105, 263, 607, 1329, 2813    & A393341 \\
        \hline
        $Av(123,132)$ & $\frac{2 x^{2} - 4 x + 1}{\left(x - 1\right) \left(4 x - 1\right)}$                                                                            & 1, 1, 3, 11, 43, 171, 683, 2731, 10923, 43691 & A007583 \\
        $Av(132,312)$ &                                                                                                                                                &                                               &         \\
        \hline
        $Av(132,213)$ & $\frac{x^{4} - 6 x^{3} + 8 x^{2} - 5 x + 1}{\left(x^{2} - 4 x + 1\right) \left(2 x^{2} - 2 x + 1\right)}$                                      & 1, 1, 3, 11, 42, 159, 596, 2225, 8300, 30967  & A393343 \\
        \hline
        $Av(123,231)$ & $\frac{12 x^{6} - 56 x^{5} + 96 x^{4} - 86 x^{3} + 41 x^{2} - 10 x + 1}{\left(x - 1\right)^{2} \left(2 x - 1\right)^{3} \left(3 x - 1\right)}$ & 1, 1, 3, 11, 41, 145, 483, 1531, 4677, 13925  & A393344 \\
        \hline
    \end{tabular}
    \caption{Generating functions, first 10 terms and OEIS entry for all symmetry classes of Cayley permutations avoiding two size 3 patterns which have a set of horizontal insertion encodings that form a regular language. }
    \label{tab:two-size-3-generating-functions}
\end{table}

\section{Concluding remarks}
\label{sec:conclusion}
We have presented two generalisations of the insertion encoding of permutations introduced by Albert, Linton and Ru\v{s}kuc~\cite{Albert2005} to Cayley permutations, namely the vertical and horizontal insertion encodings. In both cases, we classified which Cayley permutation classes have a set of insertion encodings which form a regular language and provided an algorithm to find their rational generating functions, an implementation for which can be found on GitHub~\cite{Bean_cperms_ins_enc_2025}.

Albert, Linton and Ru\v{s}kuc~\cite{Albert2005} also gave a classification for permutation classes with context free grammars, but we have not given such a condition for Cayley permutations. Our motivation is to enumerate Cayley permutations classes and their classification did not provide a method for automatically computing the algebraic generating functions, but this could be a direction for further research.

Our algorithm for enumerating the slot-bounded Cayley permutation classes extended the definition of tilings for permutations as introduced in Albert \textit{et al.} \cite{Albert2005} to Cayley permutations. By following similar methods to those outlined by Albert \textit{et al.}~\cite{combexp} we can extend our implementation of tilings to create strategies that can automatically enumerate many more classes of Cayley permutations, including those which do not have insertion encodings which form a regular language.

Cayley permutations are in bijection with ordered set partitions. The unordered set partitions are in bijection with a subset of Cayley permutation, referred in the literature as \emph{restricted growth functions} \cite{Jelínek2008}. Restricted growth functions are the Cayley permutations with the additional condition that for all values $k$ and $\ell$ in the Cayley permutation, if $k < \ell$ then the first occurrence of $k$ appears before the first occurrence of $\ell$. Our algorithm can be adapted to compute the rational generating functions for pattern-avoiding restricted growth functions with either a horizontal or vertical set of insertion encodings which are regular by restricting the letters that are used. It should be possible to classify these regular cases in a similar way as we have done for Cayley permutations, but this is left for future work.

We thank the anonymous referee for their thoughtful comments which have improved the exposition of our results.

\bibliographystyle{apalike}
\bibliography{paper.bib}


\end{document}